\documentclass[12pt]{amsart}
\usepackage{amssymb}
\usepackage{amsmath}
\newcommand\al{\alpha}
\newcommand\Af{\mathfrak{A}}%{\operatorname{\sf A}}
\newcommand\AND{\quad\text{and}\quad}
\newcommand\Bb{\mathsf{B}}
\newcommand\Bf{\mathfrak{B}}
\newcommand\C{\mathbb C}
\newcommand\D{\mathbb D}
\newcommand\Ccal{\mathcal{C}}
\newcommand\Cq{\mathsf{C}}
\newcommand\de{\delta}
\newcommand\deb{\boldsymbol{\delta}}

\newcommand\ep{\varepsilon}
\newcommand\epb{\boldsymbol{\varepsilon}}
\newcommand\Ex{\mathsf{E}}
\newcommand\Gp{\mathfrak{G}}
\newcommand\HH{{\mathbb H}}
\newcommand\If{\mathfrak{I}}
\newcommand\im{\mathfrak{i}\,}
\newcommand\lb{\mathbf{s}}%{\boldsymbol{\ell}}
\newcommand\lcd{\operatorname{lcd}}
\newcommand\ld{\boldsymbol{\ell}}%{\mathbf{l}}
\newcommand\Ll{\mathsf{L}}
\newcommand\lp{\mathfrak{l}}
\newcommand\Lp{\mathfrak{L}}

\newcommand\msf{\mathsf m}
\newcommand\N{\mathbb N}
\newcommand\Nn{\mathbb K}
\newcommand\Prob{\mathsf{Pr}}
\newcommand\Q{\mathbb{Q}}
\newcommand\R{\mathbb R}

\newcommand\Sf{\mathfrak{S}}
\newcommand\supp{\operatorname{\sf supp}}
\newcommand\tb{\mathbf{t}}

\newcommand\uno{\mathbf{1}}
\newcommand\vb{\mathbf{v}} 
\newcommand\wh{\widehat}
\newcommand\wt{\widetilde}
\newcommand\xb{\mathbf{x}} 
\newcommand\Xx{\mathsf{X}}
\newcommand\Z{\mathbb Z}

\numberwithin{equation}{section}

\newtheoremstyle{mythm}% name
  {9pt}%      Space above, empty = `usual value'
  {9pt}%      Space below
  {\itshape}% Body font
  {0pt}%         Indent amount (empty = no indent, \parindent = paraindent)
  {\bfseries}% Thm head font
  {}%        Punctuation after thm head
  { }% Space after thm head: \newline = linebreak
  {\thmnumber{(#2)}\thmname{ #1}\thmnote{ #3}}%         Thm head spec

\newtheoremstyle{mydef}% name
  {9pt}%      Space above, empty = `usual value'
  {9pt}%      Space below
  {\normalfont}% Body font
  {0pt}%         Indensf suppt amount (empty = no indent, \parindent = paraindent)
  {\bfseries}% Thm head font
  {}%        Punctuation after thm head
  { }% Space after thm head: \newline = linebreak
  {\thmnumber{(#2)}\thmname{ #1}\thmnote{ #3}}%         Thm head spec

\theoremstyle{mythm}
\newtheorem{thm}[equation]{Theorem.}
\newtheorem{pro}[equation]{Proposition.}
\newtheorem{lem}[equation]{Lemma.}
\newtheorem{cor}[equation]{Corollary.}

\theoremstyle{mydef}
\newtheorem{dfn}[equation]{Definition.}
\newtheorem{exa}[equation]{Example.}

\newtheorem{rmk}[equation]{Remark.}
\newtheorem{rmks}[equation]{Remarks.}

\begin{document}$\,$ \vspace{-1truecm}
\title{Stochastic dynamical systems with weak contractivity properties}
\author{\bf Marc PEIGN\'E and Wolfgang WOESS \\ \ \\
With a chapter featuring results  of Martin BENDA}
\address{\parbox{.8\linewidth}{Laboratoire de Math\'ematiques
et Physique Th\'eorique\\
Universit\'e Francois-Rabelais Tours\\
F\'ed\'eration Denis Poisson -- CNRS\\
Parc de Grandmont, 37200 Tours, France\\}}
\email{peigne@lmpt.univ-tours.fr}
\address{\parbox{.8\linewidth}{Institut f\"ur Mathematische Strukturtheorie 
(Math C),\\ 
Technische Universit\"at Graz,\\
Steyrergasse 30, A-8010 Graz, Austria\\}}
\email{woess@TUGraz.at}
\date{\today} 
\thanks{The first author acknowledges support by a visiting professorship
at TU Graz. The second author acknowledges support by a visiting professorship
at Universit\'e de Tours and by the Austrian Science Fund project FWF-P19115-N18.}

\begin{abstract}
Consider a proper metric space $\Xx$ and a sequence $(F_n)_{n\ge 0}$ of
i.i.d. random continuous mappings $\Xx \to \Xx$. It induces the stochastic 
dynamical system (SDS) $X_n^x = F_n \circ \dots \circ F_1(x)$ starting at 
$x \in \Xx$. In this paper, we study existence and uniqueness of invariant measures,
as well as recurrence and ergodicity of this process.

In the first part, we elaborate, improve and complete the
unpublished work of Martin Benda on local contractivity, which merits publicity
and provides an important tool for studying stochastic iterations.
We consider the case when the $F_n$ are contractions and, in particular, discuss
recurrence criteria and their sharpness for reflected random walk. 

In the second part, we consider the case where the $F_n$ are Lipschitz mappings.
The main results concern the case when the associated Lipschitz constants
are log-centered. Prinicpal tools are the Chacon-Ornstein theorem and
a hyperbolic extension of the space $\Xx$ as well as the process $(X_n^x)$.

The results are applied to the reflected affine stochastic recursion
given by $X_0^x=x \ge 0$ and $X_n^x = |A_nX_{n-1}^x - B_n|$, where
$(A_n,B_n)$ is a sequence of two-dimensional 
i.i.d. random variables with values in $\R^+_* \times \R^+_*$.
\end{abstract}

%\setcounter{tocdepth}{1}
%\tableofcontents

\maketitle

\markboth{{\sf M. Peign\'e and W. Woess}}
{{\sf Stochastic dynamical systems}}
\baselineskip 15pt

\vspace*{-.5cm}
\setcounter{tocdepth}{1}
\tableofcontents
\vspace*{-1cm}\newpage
%%%%%%%%%%%%%%%%%%%%%%%%%%%%%%%%%%%
\section{Introduction}\label{sec:intro}

We start by reviewing two well known models.  

First, let $(B_n)_{n \ge 0}$ be a sequence of i.i.d. real valued random variables.
Then \emph{reflected random walk} starting at $x \ge 0$ is the stochastic
dynamical system given recursively by $X_0^x = x$ and 
$X_n^x = |X_{n-1}^x - B_n|$. The absolute value becomes meaningful
when $B_n$ assumes positive values with positive probability; otherwise
we get an ordinary random walk on $\R$. 
Reflected random walk was described and studied by {\sc Feller~\cite{Fe}}; 
apparently, it was first considered by {\sc von Schelling~\cite{Sch}} 
in the context of telephone networks. 
In the case when $B_n \ge 0$, {\sc Feller~\cite{Fe}} and 
{\sc Knight~\cite{Kn}} have computed an invariant measure  for the 
process when the $Y_n$ are non-lattice random variables, while 
{\sc Boudiba~\cite{Bo1}}, \cite{Bo2} has provided such a measure
when the $Y_n$ are lattice variables. {\sc Leguesdron~\cite{Le}}, 
{\sc Boudiba~\cite{Bo2}} and {\sc Benda~\cite{Be}} have also studied its 
uniqueness (up to constant factors). When that invariant measure has
finite total mass -- which holds if and only if $\Ex(B_1) < \infty$ -- the 
process is (topologically) recurrent: with probability $1$, 
it returns infinitely often to each open set that is charged by the invariant
measure. Indeed, it is positive recurrent in the sense that the mean return
time is finite. More general recurrence criteria were provided by
{\sc Smirnov~\cite{Sm}} and {\sc Rabeherimanana~\cite{Rab}}, and also in
our unpublished paper \cite{PeWo}: basically, recurrence holds when
$\Ex\bigl(\sqrt{B_1\,}\bigr)$ or  quantities of more or less the same order
are finite. 
In the present paper, we shall briefly touch the situation when the
$B_n$ are not necessarily positive.

Second, let $(A_n,B_n)_{n \ge 0}$ be a sequence of i.i.d. random variables
in $\R^+_* \times \R$. (We shall always write  
$\R^+= [0\,,\,\infty)$ and $\R^+_* = (0\,,\,\infty)$, the latter usually 
seen as a multiplicative group.) %, where $\R^+_* = (0\,,\,\infty)$. 
The associated \emph{affine stochastic
recursion} on $\R$ is given by $Y_0^x = x \in \R$ and $Y_n^x = A_nY_{n-1}^x + B_n\,$.
There is an ample literature on this process, which can be interpreted
in terms of a random walk on the affine group. That is, one applies
products of affine matrices: 
$$
\begin{pmatrix} Y_n^x\\ 1 \end{pmatrix}=  
\begin{pmatrix} A_n&B_n\\ 0&1 \end{pmatrix}  
\begin{pmatrix} A_{n-1}&B_{n-1}\\ 0&1 \end{pmatrix} 
\cdots \begin{pmatrix} A_1&B_1\\ 0&1 \end{pmatrix}  
\begin{pmatrix} x\\ 1 \end{pmatrix}. 
$$
Products of affine transformations were one of the first 
examples of random walks on non-commutative groups, see 
{\sc Grenander~\cite{Gr}}.
Among the large body of further work, we mention 
{\sc Kesten}~\cite{Kes}, 
%{\sc Molchanov}~\cite{Mo}, 
{\sc Grincevi\v cjus}~\cite{G1}, \cite{G2}, 
{\sc Elie}~\cite{E1}, \cite{E2}, \cite{E3}, %, {\sc Raugi}~\cite{Ra}
 and in particular the papers by {\sc Babillot, Bougerol and Elie}~\cite{BBE}
and {\sc Brofferio}~\cite{Br}. See also the more recent work of 
%{\sc Buraczewski, Damek, Guivarc'h, Hulanicki and Urban}
{\sc Buraczewski}~\cite{Bu} and 
{\sc Buraczewski, Damek, Guivarc'h, Hulanicki and Urban}~\cite{BDGHU}.

\medskip

As an application of the results of the present paper, we shall study the 
synthesis of the above two processes.
This is the variant of the affine recursion which is forced to stay 
non-negative: whenever it reaches the negative half-axis, its sign is changed.
Thus, we have i.i.d. random variables $(A_n,B_n)_{n \ge 0}$ in $\R^+_* \times \R$,
and our process is
\begin{equation}\label{eq:aff-ref}
X_0^x = x \ge 0 \AND X_n^x = |A_nX_{n-1}^x - B_n|\,.
\end{equation}
We choose the minus sign in the recursion in order to underline the analogy
with reflected random walk. Here, we shall only consider the most typical 
situation, where $B_n > 0$. When $A_n \equiv 1$ then we are back at reflected 
random walk. 

%%%A simple family of key examples is given as follows.
%%%
%%%\begin{exa}\label{ex:simple} 
%%%Let $a > 1$ and let 
%%%$$
%%%A_n = \begin{cases} a \;&\text{ with probability }\; p\,, \\
%%%     1/a \;& \text{ with probability }\; q=1-p\,,
%%%\end{cases}
%%%     \qquad 
%%%B_n=1 \; \text{ always.}  
%%%$$
%%%Thus, we randomly iterate the transformations $f_a(x) = |ax-1|$ and 
%%%$f_{1/a}(x) = |x/a-1|$, at each step choosing $f_a$ or $f_{1/a}$ with probability
%%%$p$, resp. $q$. When $1 < a \le 2$, both transformations map the interval
%%%$[0\,,\,1]$ into itself, so that the process is best considered on that
%%%interval. When $a > 2$, it evolves on $[0\,,\infty)$.
%%%\end{exa}

In all those introductory examples, the hardest and most interesting case is the 
one when $A_n$ is \emph{log-centered,} that is, $\Ex(\log A_n) = 0$, and the 
development of
tools for handling this case is the main focus of the present work.
The easier and well-understood case is the \emph{contractive} one, where
 $\Ex(\log A_n) < 0$.

In this paper, stochastic dynamical systems are considered in
the following general setting. Let $(\Xx,d)$ be a proper metric space 
(i.e., closed balls are compact), and let $\Gp$ be the monoid of all 
continuous mappings $\Xx \to \Xx$.
It carries the topology of uniform convergence on compact sets.
Now let $\wt\mu$ be a regular probability measure on $\Gp$, and let 
$(F_n)_{n \ge 1}$ be a sequence of i.i.d. $\Gp$-valued random variables
(functions) with common distribution $\wt\mu$, defined on a suitable 
probability space $(\Omega, \Af, \Prob)$. 
The measure $\wt\mu$ gives rise to the \emph{stochastic dynamical system
(SDS)} $ \omega \mapsto X_n^x(\omega)$ defined by
\begin{equation}\label{eq:SDS}
X_0^x = x \in \Xx\,,\AND X_n^x = F_n(X_{n-1}^x)\,,\quad n \ge 1\,.
\end{equation}
There is an ample literature on processes of this type, see e.g. 
{\sc Arnold}~\cite{Ar} or {\sc Bhattacharya and Majumdar}~\cite{Bh-Ma}.
In the setting of our reflected affine recursion \eqref{eq:aff-ref}, 
we have $\Xx = \R^+$ with the standard distance, and $F_n(x) = |A_n x-B_n|$, so 
that the measure $\wt\mu$ is the image of the distribution $\mu$ of the 
two-dimensional i.i.d. random variables $(A_n,B_n)$ under the mapping
$\R \times \R^+_*\to \Gp\,$, $(a,b) \mapsto f_{a,b}\,$, where 
$f_{a,b}(x) = |ax-b|$. Any SDS \eqref{eq:SDS} is a Markov chain. The transition
kernel is
$$
P(x,U) = \Prob[X_1^x \in U] = \wt\mu (\{ f \in \Gp : f(x) \in U \})\,,
$$
where $U$ is a Borel set in $\Xx$. The associated transition operator is
given by 
$$
P\varphi(x) = \int_{\Xx} \varphi(y)\,P(x,dy) = \Ex\bigl(\varphi(X_1^x)\bigr)\,,
$$
where $\varphi:\Xx \to \R$ is a measurable function for which this integral
exists. The operator is Fellerian, that is, $P\varphi$ is continuous when
$\varphi$ is bounded and continuous. 
We shall write $\Ccal_c(\Xx)$ for the space of compactly supported continuous
functions $\Xx \to\R$.

The SDS is called \emph{transient,} if every compact set is visited only finitely
often, that is,
$$
\Prob[d(X_n^x,x) \to \infty] = 1 \quad\text{for every}\;x \in \Xx.
$$
We call it \emph{(topologically) recurrent,} if there is a non-empty, closed
set $\Ll \subset \Xx$ such that for every open set $U$ that intersects $\Ll$,
$$
\Prob[X_n^x \in U \;\text{infinitely often}] = 1 \quad\text{for every}\;x \in \Ll.   
$$
In our situation, we shall even have this for every starting point $x \in \Xx$, 
so that $\Ll$ is an \emph{attractor} for the SDS.
As an intermediate notion, we call the SDS \emph{conservative,} if 
$$
\Prob[\liminf\nolimits_n d(X_n^x,x) < \infty] = 1 \quad\text{for every}\;x \in \Xx.
$$
Besides the question whether the SDS is recurrent, we shall mainly be interested
in the question of existence and uniqueness (up to constant factors) 
of an \emph{invariant measure.} This is a Radon measure $\nu$ on $\Xx$
such that for any Borel set $U \subset \Xx$,
$$
\nu(U) = \int_{\Xx} \Prob[X_1^x \in U]\,d\nu(x)\,.
$$
We can construct the \emph{trajectory space} of the SDS starting at $x$. This is
$$
\bigl(\Xx^{\N_0}, \Bf(\Xx^{\N_0}), \Prob_x\bigr),
$$
where $\Bf(\Xx^{\N_0})$ is the product Borel $\sigma$-algebra on $\Xx^{\N_0}$, 
and $\Prob_x$ is the image of the measure $\Prob$ under the mapping
$$
\Omega \to \Xx^{\N_0}\,,\quad \omega \mapsto 
\bigl( X_n^x(\omega) \bigr)_{n \ge 0}\,.
$$
If we have an invariant Radon measure, then we can construct the measure
$$
\Prob_{\nu} = \int_{\Ll} \Prob_x \, d\nu(x)
$$
on the trajectory space.
It is a probability measure only when $\nu$ is a probability measure on
$\Xx$. In general, it is $\sigma$-finite and invariant with respect to the
time shift $T: \Xx^{\N_0} \to \Xx^{\N_0}$. Conservativity of the
SDS will be used to get conservativity of the shift. We shall study ergodicity
of $T$, which in turn will imply uniqueness of $\nu$ (up to multiplication
with constants).

As often in this field, ideas that were first developped 
by {\sc Furstenberg},  e.g. \cite{Fu}, play an important role at least in
the background. 

\begin{pro}[{[Furstenberg's contraction principle.]}]\label{pro:furst}
Let $(F_n)_{n \ge 1}$ be i.i.d. continuous random mappings $\Xx \to \Xx$,
and define the right process
$$
R_n^x = F_1 \circ \dots \circ F_n(x)\,.
$$
If there is an $\Xx$-valued random variable $Z$ such
that 
$$
\lim_{n\to \infty} R_n^x = Z \quad\text{almost surely for every }\;x \in \Xx\,,
$$
then the distribution $\nu$ of the limit $Z$ is the unique invariant 
probability measure for the SDS $X_n^x = F_n \circ \dots \circ F_1(x)$.
%, which is (positive) recurrent
%on $\Ll = \supp(\nu)$. 
\end{pro}

A proof can be found, e.g., in {\sc Letac}~\cite{Let} in a slightly
more general setting. 

While being ideally applicable to the contractive case,
this contraction principle is not the right tool for handling the log-centered
case mentioned above. In the context  of the affine stochastic recursion,
{\sc Babillot, Bougerol and Elie~\cite{BBE}} introduced the notion of 
\emph{local contractivity,}
see Definition \ref{def:loccont} below.
This was then exploited systematically by {\sc Benda} in interesting and useful 
work in his PhD thesis \cite{Be} (in German)  and the two 
subsequent preprints  \cite{Be1}, \cite{Be2} which were accepted for
publication, circulated (not very widely) in preprint version but have remained 
unpublished. In personal comunication, {\sc Benda}
also gives credit to unpublished work
of his late PhD advisor {\sc Kellerer}, compare with the posthumous
publication \cite{Ke}. 

We think that this material deserves to be documented in a
publication, whence we include -- with the consent of M. Benda whom we managed to
contact -- the next section on weak contractivity (\S \ref{sec:benda}). 
The proofs that we give are ``streamlined'', and new aspects and results are added,
such as, in particular, ergodicity of the shift on the trajectory space with
respect to $\Prob_{\nu}\,$ (Theorem \ref{thm:uniquemeasure}). 
Ergodicity yields uniqueness of the invariant measure.
Before that, we explain the alternative between recurrence and transience and
the limit set (attractor) $\Ll$, which is the support of the invariant measure
$\nu$.

We display briefly the classical results regarding the stochastic
affine recursion in \S \ref{sec:affine}. Then, in \S \ref{sec:contract}, 
we consider the situation when the $F_n$ are contractions with Lipschitz 
constants $A_n = \lp(F_n) \le 1$ (not necessarily assuming that 
$\Ex(\log A_n) < 0$). We provide a tool for getting strong contractivity
in the recurrent case (Theorem \ref{thm:contractive}).
A typical example is reflected random walk. In \S \ref{sec:reflected},
we discuss some of its properties, in particular sharpness of 
recurrence criteria. 

This concludes Part I of the paper. In Part II, we examine in detail the
iteration of general Lipschitz mappings. That is, the Lipschitz constants
$A_n = \lp(F_n)$ of the $F_n$ are positive, finite, i.i.d. random variables.
The emphasis is on the case when the $A_n$ are log-centered. We impose natural
non-degeneracy assumptions and suitable moment conditions on $A_n$ as
well as $B_n = d\bigl(F_n(o),o\bigr)$, where $o \in \Xx$ is a reference point. 
We first prove existence of a non-empty limit set $\Ll$ on which the SDS 
is recurrent (\S \ref{sec:lip}, Theorem \ref{thm:recur}).  

Then (\S \ref{sec:hyp}) we introduce a \emph{hyperbolic extension}
of the space $\Xx$ as well as of the SDS. The extended SDS turns out to be generated
by Lipschitz mappings with Lipschitz constants $=1$ (Lemma \ref{lem:lift-f}).
The hyperbolic extension appears to be interesting in its own right, and we intend
to come back to it in future work. It yields that the extended SDS is either
transient or conservative, although in general typically not locally contractive.

First, in \S \ref{sec:trans-ext}, we consider the case when the extended SDS
is transient. In this case, we can show (\ref{thm:lip-trans}) that the original SDS
is locally contractive, so that all results of \S \ref{sec:benda} apply.
In particular, we get uniqueness of the invariant Radon measure $\nu$ (up to
constant factors) and ergodicity of the shift on the associated trajectory
space. It is worth while to mention that the ``classical'' instance of this
situation is the affine stochastic recursion. Its hyperbolic extension is a
random walk on the affine group, which is well known to be transient.

The hardest case turns out to be the one when the extended SDS is conservative
(\S \ref{sec:cons-ext}). In this case, we are able to obtain a result only
under an additional assumption \eqref{eq:assume3} on the original SDS that 
resembles the criterion used in \S \ref{sec:contract} for SDS of contractions.
But then we even get ergodicity and uniqueness of the invariant Radon measure
for the extended SDS (Theorem \ref{thm:divide}).

In the final section (\S \ref{sec:ref-aff}), we explain how to apply 
all those results to the reflected affine stochastic recursion. 

Since we want to present a sufficiently comprehensive picture, we have
included -- mostly without proof -- a few known results, in particular on
cases where one has strong contractivity. 

\newpage

\addcontentsline{toc}{section}{\\[-8pt] \hspace*{.5cm}PART I. Strong and local 
contractivity and examples, including\\ \hspace*{.5cm}reflected random walk}
\begin{center}
{\large \bf PART I. Strong and local 
contractivity and examples, including reflected random walk}
\end{center}
\section{Local contractivity and the work of Benda}\label{sec:benda}

\begin{dfn}\label{def:loccont}
\emph{(i)} The SDS is called \emph{strongly contractive,} if for 
every $x \in \Xx$, 
$$
\Prob[d(X_n^x,X_n^y)\to 0 \quad\text{for all}\; y \in \Xx] = 1\,.
$$
\emph{(ii)}
The SDS is called \emph{locally contractive,} if for 
every $x\in \Xx$ and every compact $K \subset \Xx$, 
$$
\Prob[d(X_n^x,X_n^y)\, \cdot \uno_K(X_n^x)\to 0 
\quad\text{for all}\; y \in \Xx] = 1\,.
$$
\end{dfn}

Let $\Bb(r)$ and $\overline\Bb(r)$, $r \in \N$, be the open and
closed balls in $\Xx$ with radius $r$ and fixed center $o \in \Xx$,
respectively. $\overline\Bb(r)$ is compact by properness of $\Xx$.

Using Kolmogorov's 0-1 law, one gets the following alternative. 
\begin{lem}\label{lem:transient} For a locally contractive SDS,
$$
\begin{aligned}
\text{either}\quad &\Prob[d(X_n^x,x) \to \infty]=0 
\quad\text{for all}\;\ x \in \Xx\,,\\ 
\text{or}\qquad\;\; &\Prob[d(X_n^x,x) \to \infty]=1 
\quad\text{for all}\;\ x \in \Xx\,.
\end{aligned}
$$
\end{lem} 

\begin{proof}  Consider 
\begin{equation}\label{eq:Xmn}
X_{m,m}^x= x \AND X_{m,n}^x = F_n \circ F_{n-1} \circ \ldots \circ F_{m+1}(x)
\quad \text{for}\; n>m, 
\end{equation}
so that $X_n^x =X_{0,n}^x\,$. Then local contractivity
implies that for each $x \in \Xx$, we have $\Prob(\Omega_0)=1$ for the
event $\Omega_0$ consisting of all $\omega \in \Omega$ with
\begin{equation}\label{eq:Om0}
\lim_{n\to \infty} \uno_{\Bb(r)}\bigl(X_{m,n}^x(\omega)\bigr) \cdot 
d\bigl(X_{m,n}^x(\omega),X_{m,n}^y(\omega)\bigr) = 0
\quad\text{for each}\; r \in \N\,,\; m \in \N_0\,,\; y \in \Xx. 
%\quad\text{as}\; n \to \infty\,.
\end{equation}
Clearly, $\Omega_0$ is invariant with respect to the shift of the sequence 
$(F_n)$. 
 
Let  $\omega \in \Omega_0$ be such that the sequence 
$\bigl( X_n^x(\omega)\bigr)_{n \ge 0}$ accumulates at some $z \in \Xx$.
Fix $m$ and set $v = X_m^x(\omega)$. Then also 
$\bigl(X_{m,n}^{v}(\omega)\bigr)_{n \ge m}$
accumulates at $z$. Now let $y \in \Xx$ be arbitrary. Then there is $r$ such
that $v, y, z \in \Bb(r)$. Therefore also 
$\bigl( X_{m,n}^y(\omega)\bigr)_{n \ge m}$ accumulates at $z$. 
In particular, the fact that $\bigl( X_n^x(\omega)\bigr)_{n \ge 0}$ accumulates
at some point does not depend on the initial trajectory, i.e., on the
specific realization of $F_1, \dots, F_m\,$. We infer that the set
$$
\bigl\{ \omega \in \Omega_0 : \bigl( X_n^x(\omega)\bigr)_{n \ge 0}
\; \text{accumulates in} \; \Xx \bigr\}
$$ 
is a tail event of $(F_n)_{n \ge 1}$. On its complement in $\Omega_0\,$,
we have $d(X_n^x,x) \to \infty\,$. 
\end{proof}

If $d(X_n^x,x) \to \infty$ almost surely, then we call the SDS \emph{transient.}

For $\omega \in \Omega$, let $\Ll^x(\omega)$ be the set of accumulation 
points of $\bigl( X_n^x(\omega)\bigr)$ in $\Xx$.  The following proof
is much simpler than the one in \cite{Be1}.

\begin{lem}\label{lem:attract} 
For any conservative, locally contractive SDS, there is a set 
$\Ll \subset \Xx$ -- the \emph{attractor} or \emph{limit set} -- such that
$$
\Prob[ \Ll^x(\cdot) = \Ll \;\ \text{for all}\;\ x \in \Xx] = 1\,,
$$
\end{lem}
\begin{proof}
The argument of the proof of Lemma \ref{lem:transient} also shows the
following. For every open $U \subset \Xx$, 
$$
\Prob[X_n^x \;\text{accumulates in}\; U \;\text{for all}\;\ x \in \Xx] \in 
\{0,1\}\,.
$$
$\Xx$ being proper, we can find a countable basis $\{ U_k : k \in \N \}$
of the topology of $\Xx$, where each $U_k$ is an open ball.
Let $\Nn \subset \N$ be the (deterministic) set of all $k$ such that the 
above probability is $1$ for $U=U_k\,$. Then there is 
$\Omega_0 \subset \Omega$ such that
$\Prob(\Omega_0)=1$, and for every $\omega \in \Omega_0\,$,
the sequence $\bigl(X_n^x(\omega)\bigr)_{n \ge 0}$ accumulates in $U_k$
for some and equivalently all $x$ precisely when $k \in \Nn$.  
Now, if $\omega \in \Omega_0\,,$ then $y \in \Ll^x(\omega)$ if and only if
when $k \in \Nn$ for every $k$ with $U_k \ni y$. We see that
$\Ll^x(\omega)$ is the same set for every $\omega \in \Omega_0\,$.
\end{proof} 

Thus, $(X_n^x)$ is \emph{(topologically) recurrent} on $\Ll$
when $\Prob[d(X_n^x,x) \to \infty]=0$, that is, every open set that
intersects $\Ll$ is visited infinitely often with probability $1$.

For a Radon measure $\nu$ on $\Xx$, its transform under $P$ is written as
$\nu P$, that is, for any Borel set $U \subset \Xx$,
$$
\nu P(U) = \int_{\Xx} P(x,U)\,d\nu(x)\,.
$$ 
Recall that $\nu$ is called \emph{excessive,} when $\nu P \le \nu$, and
\emph{invariant,} when $\nu P = \nu$.
%We shall write $\nu(\varphi) = \int \varphi\,d\nu$, when $\varphi$ is a $\nu$-integrable real
%function on $\Xx$.  

For two transition kernels $P, Q$, their product is defined as 
$$
PQ(x,U) = \int_{\Xx} Q(y,U)\,P(x,dy)\,. 
$$
In particular, $P^k$ is the $k$-fold iterate.
The first part of the following is well-known; we outline the proof because
it is needed in the second part, regarding $\supp(\nu)$.

\begin{lem}\label{lem:exc-inv} 
If the locally contractive SDS is recurrent, then every excessive
measure $\nu$ is invariant. Furthermore, $\supp(\nu) = \Ll$.
\end{lem}

\begin{proof}
For any pair of Borel sets $U, V \subset X$, define the transition
kernel $P_{U,V}$  and the measure $\nu_U$ by
$$
P_{U,V}(x,B) = \uno_U(x)\,P(x,B \cap V) \AND \nu_U(B) = \nu(U \cap B)\,,
$$
where $B \subset \Xx$ is a Borel set. We abbreviate $P_{U,U} = P_U$.
Also, consider the stopping time $\tau_x^U = \inf \{ n \ge 1: X_n^x \in U \}$,
and for $x \in U$ let
$$
P^U(x,B) = \Prob[ \tau_x^U < \infty\,,\; X_{\tau_x^U}^x \in B ]
$$
be the probability that the first return of $X_n^x$ to the set $U$ occurs in a
point of $B \subset X$. Then we have 
$$
\nu_U \ge \nu_U\,P_U + \nu_{U^c}\, P_{U^c,U}\,,
$$
and by a typical inductive (``balayage'') argument,
$$
\nu_U \ge 
\nu_U \left( P_U + \sum_{k=0}^{n-1} P_{U,U^c} \, P_{U^c}^k \,P_{U^c,U}\right)  
+ \nu_{U^c}\,P_{U^c}^n\, P_{U^c,U}\,.
$$
In the limit, 
$$
\nu_U \ge \nu_U 
\left( P_U + \sum_{k=0}^{\infty} P_{U,U^c} \, P_{U^c}^k \,P_{U^c,U}\right)
= \nu_U\, P^U\,.
$$
Now suppose that $U$ is open and relatively compact, and 
$U \cap \Ll \ne \emptyset$. Then, by
recurrence, for any $x \in U$, we have $\tau_x^U < \infty$ almost surely.
This means that $P^U$ is stochastic, that is, $P^U(x,U) =1$.
But then $\nu_U\,P^U(U) = \nu_U(U) = \nu(U) < \infty$. Therefore
$\nu_U = \nu_U\,P^U$.  
We now can set $U = \Bb(r)$ and let $r \to \infty$. Then monotone convergence
implies $\nu = \nu P$, and $P$ is invariant.

Let us next show that $\supp(\nu) \subset \Ll$. 

Take an open, relatively compact set $V$ such that $V \cap \Ll = \emptyset$.

Now choose $r$ large enough such that $U=\Bb(r)$ contains $V$ and intersects 
$\Ll$. Let $Q = P^U$. We know from the above that 
$\nu_U = \nu_U\,Q = \nu_U\,Q^n$. We get
$$
\nu(V) = \nu_U(V) = \int_{U} Q^n(x,V)\,d\nu_U(x)\,.
$$
Now $Q^n(x,V)$ is the probability that the SDS starting at $x$ visits
$V$ at the instant when it returns to $U$ for the $n$-th time.  
As
$$
\Prob[X_n^x \in V \;\text{for infinitely many}\;n] =0\,,
$$
it is an easy exercise to show that $Q^n(x,V) \to 0$. Since the measure
$\nu_U$ has finite total mass, we can use dominated convergence to
see that $\int_{U} Q^n(x,V)\,d\nu_U(x) \to 0$ as $n \to \infty$.

We conclude that $\nu(V) = 0$, and $\supp(\nu) \subset \Ll$.

Since $\nu P = \nu$, we have $f\bigl(\supp(\nu)\bigr) \subset \supp(\nu)$
for every $f \in \supp(\wt\mu)$, where (recall) $\wt \mu$ is the distribution 
of the random functions $F_n$ in $\Gp$. But then almost surely 
$X_n^x \in \supp(\nu)$ for all $x \in \supp(\nu)$ and all $n$,
that is, $\Ll^x(\omega) \subset \supp(\nu)$ for $\Prob$-almost every $\omega$.
Lemma \ref{lem:attract} yields that $\Ll \subset \supp(\nu)$. 
\end{proof} 

The following holds in more generality 
than just for recurrent locally contractive SDS.

\begin{pro}\label{pro:invmeasure} 
If the locally contractive SDS is recurrent, then it possesses an
invariant measure $\nu$.
\end{pro}

\begin{proof}
Fix $\psi \in \Ccal_c^+(\Xx)$ such that its support intersects $\Ll$.
Recurrence implies that
$$
\sum_{k=1}^{\infty} P^k\psi(x) = \infty \quad \text{for every}\; x\in \Xx. 
$$
The statement now follows from a result of {\sc Lin}~\cite[Thm. 5.1]{Li}.
\end{proof}

Thus we have an invariant Radon measure $\nu$ with $\nu P = \nu$ and
$\supp(\nu) = \Ll$. It is now easy to see that the attractor depends only 
on $\supp (\wt \mu) \subset \Gp$.

\begin{cor}\label{cor:support}
In the recurrent case, $\Ll$ is the smallest non-empty closed subset of 
$\Xx$ with the
property that $f(\Ll) \subset \Ll$ for every $f \in \supp (\wt \mu)$.
\end{cor}

\begin{proof} The reasoning at the end of the proof of Lemma \ref{lem:exc-inv} 
shows that $\Ll$ is indeed a closed
set with that property. On the other hand, if $C \subset \Xx$ is closed,
non-empty and such that $f(C) \subset C$ for all $f \in \supp(\wt\mu)$
then $\bigl(X_n^x(\omega)\bigr)$ evolves almost surely within $C$ 
when the starting point $x$ is in $C$. But then $\Ll^x(\omega) \subset C$ 
almost surely, and on the other hand $\Ll^x(\omega) =\Ll$ almost surely.
\end{proof}

\begin{rmk}\label{rem:non-contractive}
Suppose that the SDS induced by the probability measure $\wt \mu$ on $\Gp$
is not necessarily locally contractive, resp. recurrent, but that there is 
another probability measure $\wt \mu'$ on $\Gp$ which does induce a 
weakly contractive, recurrent SDS and
which %is equivalent with $\wt \mu$ in the sense of mutual absolute continuity. 
satisfies $\supp(\wt\mu)=\supp(\wt\mu')$. 
Let $\Ll$ be the limit
set of this second SDS. Since it depends only on $\supp(\wt\mu')$,
the results that we have so far yield that also for the SDS
$(X_n^x)$ associated with $\wt \mu$, $\Ll$ is the unique ``essential class'' 
in the following sense: it is the unique minimal non-empty closed subset of 
$\Xx$ such that
\\[3pt]
(i) for every open set $U \subset \Xx$ that intersects $\Ll$
and every starting point $x \in \Xx$, the sequence $(X_n^x)$ visits
$U$ with positive probability, and
\\[3pt]
(ii) if $x \in \Ll$ then $X_n^x \in \Ll$ for all $n$. \hfill$\square$
\end{rmk}

For $\ell \ge 2$, we can lift each $f \in \Gp$ to a continuous mapping 
$$
f^{(\ell)}:\Xx^{\ell} \to \Xx^{\ell}\,,\quad 
f^{(\ell)}(x_1, \dots, x_{\ell}) = \bigl( x_2, \dots, x_d, f(x_{\ell})\bigr)\,.
$$
In this way, the random mappings $F_n$ induce the SDS 
$\bigl(F_n^{(\ell)}\circ \dots \circ F_1^{(\ell)}(x_1, \dots, x_{\ell})\bigr)_{n \ge 0}$
on $\Xx^{\ell}$. For $n \ge \ell-1$ this is just 
$\bigl(X_{n-\ell+1}^{x_\ell}, \dots, X_n^{x_{\ell}}\bigr)\,$.

\begin{lem}\label{lem:repeat} Let $x \in \Xx$, and let $U_0, \dots, U_{\ell-1}
\subset \Xx$ be Borel sets such that
$$
\begin{aligned}
&\Prob[X_n^x \in U_0 \;\;\text{for infinitely many}\;\, n] =1 \AND\\
&\Prob[X_1^y \in U_j] \ge \al > 0 \quad\text{for every}\;\; y \in U_{j-1}\,,
\; j=1, \dots, \ell-1.
\end{aligned}
$$
Then also 
$$ 
\Prob[X_n^x \in U_0\,,\; X_{n+1}^x \in U_1\,,\dots, X_{n+\ell-1}^x \in U_{\ell-1} 
\;\;\text{for infinitely many}\;\, n] =1.
$$
\end{lem}

\begin{proof}
This is quite standard and true for general Markov chains and not just SDS.
Let $\tau(n)$, $n \ge 1$, be the stopping times of the successive visits
of $(X_n^x)$ in $U$. They are all a.s. finite by assumption. We consider
the events 
$$
\Lambda_n = [X_{\tau(\ell n)+1}^x \in U_1\,,\dots, 
X_{\tau(\ell n)+\ell-1}^x \in U_{\ell-1}]
\AND \Lambda_{k,m} = \textstyle{\bigcup_{n = k+1}^{m-1}} \Lambda_n\,,
$$
where $k < m$.
We need to show that $\Prob(\limsup_n \Lambda_n) =1$. 
By the strong Markov property, we have 
$$
\Prob(\Lambda_n \mid X_{\tau(\ell n)}^x = y) \ge \alpha^{\ell} \quad
\text{for every}\; y \in U_0\,.
$$ 
Let $k, m\in \N$ with $k < m$. Just for the purpose of the next lines of
the proof, consider the measure on $\Xx$ defined by
$$
\sigma(B) = \Prob\bigl([X_{\tau(\ell m)}^x\in B] 
\cap \Lambda_{k,m-1}^c\bigr).
$$
It is concentrated on $U$, and using the Markov property,
$$
\begin{aligned} 
\Prob(\Lambda_{k,m}^c)
&= \int_U \Prob(\Lambda_m^c \,|\, X_{\tau(\ell m)}^x = y) \,d\sigma(y) \\
&\le (1-\alpha^{\ell})\, \sigma(U)
=(1-\alpha^{\ell})\,\Prob(\Lambda_{k,m-1}^c) \le \dots \le
(1-\alpha^{\ell})^{m-k}\,.
\end{aligned}
$$
Letting $m \to \infty$, we see that  
$\;\Prob\bigl(\textstyle{\bigcap_{n > k}} \Lambda_n^c\bigr) = 0\;$ for every $k$,
so that
$$
\Prob\bigl(\textstyle{\bigcap_{k}\bigcup_{n > k}} \Lambda_n\bigr) = 1,
$$
as required.
\end{proof}

\begin{pro}\label{pro:l-process}
If the SDS is locally contractive and recurrent on $\Xx$, then so is the 
lifted process on $\Xx^{\ell}$. The limit set of the latter is
$$
\Ll^{(\ell)} = \bigl\{ \bigl(x, f_1(x), f_2\circ f_1(x), \dots, 
      f_{\ell-1}\circ \dots f_1(x)\bigr) : x \in \Ll, \; f_i 
      \in \supp (\wt\mu)\bigr\}^-\,,
$$        
and if the Radon measure $\nu$ is invariant for the original SDS on $\Xx$,
then the measure $\nu^{(\ell)}$ is invariant for the lifted SDS on $\Xx^{\ell}$,
where
$$
\int_{\Xx^{\ell}} f\,d\nu^{(\ell)} = \int_{\Xx}\cdots \int_{\Xx} 
f(x_1, \dots, x_{\ell})\, P(x_{\ell-1},dx_{\ell})\, 
P(x_{\ell-2},dx_{\ell-1}) \cdots P(x_1,dx_2)\,d\nu(x_1)\,.
$$  
\end{pro}

\begin{proof} It is a straightforward exercise to verify that the lifted SDS is
locally contractive and has  $\nu^{(\ell)}$ as an invariant measure. We have to
prove that  it is recurrent. For this purpose, we just have to show that there is 
some relatively compact subset of $\Xx^{\ell}$ that is visited infinitely often
with positive probability. We can find relatively compact open subsets 
$U_0\,, \dots, U_{\ell-1}$ of $\Xx$ that intersect $\Ll$ such that 
$$
\Prob[ F_1(U_{j-1}) \subset U_j ] \ge \alpha > 0\quad
\text{for}\; j=1, \dots, \ell-1\,.
$$
We know that for arbitrary starting point $x \in \Xx$, with probability 1, the
SDS $(X_n^x)$ visits $U_0$ infinitely often. Lemma \ref{lem:repeat} implies that
the  lifted SDS on $\Xx^{\ell}$ visits 
$U_0 \times \dots \times U_{\ell-1}$ infinitely often with probability 1. 

By Lemma \ref{lem:transient},
the lifted SDS on $\Xx^{\ell}$ is recurrent. 
Now that we know this, it is clear from Corollary \ref{cor:support}
that its attractor is the set $\Ll^{\ell}$, as stated.
\end{proof}

%In view of Lemma \ref{cor:support}, the SDS evolves within $\Ll$
%when started at some point of $\Ll$, and every invariant measure is supported
%in that set.  
%
%Thus, we can consider the
%random mappings $F_n$ and their distribution $\wt \mu$ just on $\Ll$ in the
%place of the whole of $\Xx$. Recall the probability space 
%$(\Omega, \Af, \Pr)$  on which the $\Gp$-valued random variables 
%$F_n$ are defined, so that (more precisely) 
%$F_n(x) = F_n(\omega,x)$ for $x \in \Ll$, $\omega \in \Omega$.

As outlined in the introduction, we can equip the trajectory space $X^{\N_0}$
of our SDS with the infinite product $\sigma$-algebra and the
measure $\Prob_{\nu}\,$, which is in general $\sigma$-finite.

\begin{lem}\label{lem:conservative} If the SDS is locally contractive
and recurrent, then $T$ is conservative on 
$\bigl(\Xx^{\N_0}, \Bf(\Xx^{\N_0}), \Prob_{\nu}\bigr)$. 
\end{lem}

\begin{proof}
Let $\varphi = \uno_U\,$, where $U \subset \Xx$ is open, relatively compact, and
intersects $\Ll$. We can extend it to a strictly positive function in 
$L^1(\Xx^{\N_0},\Prob_{\nu})$ by setting $\varphi(\xb) = \varphi(x_0)$ for
$\xb = (x_n)_{n\ge 0}\,$. We know from recurrence that 
$$
\sum_n \varphi(X_n^x) = \infty \quad \Prob\text{-almost surely, for every}\;
x \in \Xx\,.
$$
This translates into
$$
\sum_n \varphi(T^n\xb) = \infty \quad \Prob_{\nu}\text{-almost surely, for every}\;
\xb \in \Xx^{\N_0}\,.
$$
Conservativity follows; see e.g. \cite[Thm. 5.3]{Rev}.
\end{proof}

The uniqueness part of the following theorem is contained in 
\cite{Be} and \cite{Be1}; see also 
{\sc Brofferio~\cite[Thm. 3]{Br}}, who considers SDS of affine mappings. 
We modify and extend the proof in order to be
able to conclude that our SDS is ergodic with respect to $T$. (This, as well as
Proposition \ref{pro:l-process}, is new with respect to Benda's work.)

\begin{thm}\label{thm:uniquemeasure}
For a recurrent locally contractive SDS, let $\nu$ be the measure 
of Proposition \ref{pro:invmeasure}. Then the 
shift $T$ on $\Xx^{\N_0}$is ergodic with respect to $\Prob_{\nu}\,$.

In particular, $\nu$ is the unique invariant Radon 
measure for the SDS up to multiplication with constants. 
\end{thm}  

\begin{proof} 
Let $\If$ be the $\sigma$-algebra of the
$T$-invariant sets in $\Bf(\Xx^{\N_0})$.
For $\varphi \in L^1(\Xx^{\N_0},\Prob_{\nu})$, we write
$\Ex_{\nu}(\varphi) = \int \varphi \,d\Prob_{\nu}$ and
$\Ex_{\nu}(\varphi\,|\,\If)$ for the conditional ``expectation'' of
$\varphi$ with respect to $\If$. The quotation marks refer to the fact that
it does not have the meaning of an expectation when $\nu$ is not a probability
measure. As a matter of fact, what is well defined in the latter case are quotients
$\Ex_{\nu}(\varphi\,|\,\If)/\Ex_{\nu}(\psi\,|\,\If)$ for suitable 
$\psi \ge 0$; compare with the explanations in 
{\sc Revuz}~\cite[pp. 133--134]{Rev}.

In view of Lemma \ref{lem:conservative}, we can apply
the ergodic theorem of  {\sc Chacon and Ornstein}~\cite{ChOr}, see also
\cite[Thm.3.3]{Rev}. 
Choosing an arbitrary function $\psi \in L^1(\Xx^{\N_0},\Prob_{\nu})$ 
with
\begin{equation}\label{eq:g-cond}
\Prob_{\nu}\Bigl(\Bigl\{ \xb \in \Xx^{\N_0}: 
\sum_{n=0}^{\infty} \psi(T^n\xb) < \infty\Bigr\}\Bigr) =0,
\end{equation}
one has for every $\varphi \in L^1(\Xx^{\N_0},\Prob_{\nu})$
\begin{equation}\label{eq:ChacOrn-1}
\lim_{n \to \infty} 
\frac{\sum_{k=0}^n \varphi(T^k\xb)}{\sum_{k=0}^n \psi(T^k\xb)} 
= \frac{\Ex_{\nu}(\varphi\,|\,\If)}{\Ex_{\nu}(\psi\,|\,\If)} \quad \text{for }\;
\Prob_{\nu}\text{-almost every }\; \xb \in \Xx^{\N_0}. 
\end{equation}
In order to show ergodicity of $T$, we need to show that
the right hand side is just 
$$
\frac{\Ex_{\nu}(\varphi)}{\Ex_{\nu}(\psi)}\,.
$$
It is sufficient to show this for non-negative functions that depend only on 
finitely many coordinates. For a function $\varphi$ on $\Xx^{\N_0}$, we also
write $\varphi$ for its extension to $\Xx^{\N_0}$, given
by $\varphi(\xb) = \varphi(x_0,\dots, x_{\ell-1})$.
   
That is, we need to show that for every $\ell \ge 1$ and
non-negative Borel functions $\varphi, \psi$ on $\Xx^{\ell}$,
with  $\psi$ satisfying \eqref{eq:g-cond},
\begin{equation}\label{eq:ChacOrn-2}
\begin{aligned}
&\lim_{n \to \infty} 
\frac{\sum_{k=0}^n \varphi\bigl(X_k^x(\omega), \dots, X_{k+\ell-1}^x(\omega)\bigr)}
{\sum_{k=0}^n \psi\bigl((X_k^x(\omega), \dots, X_{k+\ell-1}^x(\omega)\bigr))} 
= \frac{\int_{\Ll} \Ex\bigl(\varphi(X_0^y, \dots, X_{\ell-1}^y)\bigr)\,d\nu(y)}
{\int_{\Ll} \Ex\bigl(\psi(X_0^y, \dots, X_{\ell-1}^y)\bigr)\,d\nu(y)}\\[4pt]
&\text{for}\; 
\nu\text{-almost every}\; x \in \Xx\;\text{and}\; \Prob\text{-almost every}
\; \omega \in \Omega,
\end{aligned}
\end{equation}
when the integrals appearing in the right hand term are finite.

\smallskip

At this point, we observe that we need to prove \eqref{eq:ChacOrn-2} only 
for $\ell=1$. 
Indeed, once we have the proof for this case, we can reconsider our SDS 
on $\Xx^{\ell}$, and using Propostion \ref{pro:l-process}, our proof 
for $\ell=1$ applies to the new SDS as well.

\smallskip

So now let $\ell=1$. By regularity of $\nu$, we may assume that 
$\varphi$ and $\psi$ are non-negative, compactly supported, 
continuous functions on $\Ll$ that both are non-zero.

We consider the random variables 
$S_n^x\varphi(\omega) = \sum_{k=0}^n \varphi\bigl(X_k^x(\omega)\bigr)$ 
and $S_n^x\psi(\omega)$. Since the SDS is recurrent, both functions satisfy 
\eqref{eq:g-cond}, i.e., we have almost surely that $S_n^x\varphi$ and 
$S_n^x\psi > 0$ for all but finitely many $n$ and all $x$. We shall show that 
\begin{equation}\label{eq:ratio}
\lim_{n \to \infty} \frac{S_n^x\varphi}{S_n^x\psi} 
= \frac{\int_{\Ll} \varphi\,d\nu}{\int_{\Ll} \psi\,d\nu} \quad 
\Prob\text{-almost surely}\;\text{and for \emph{every}}\; x \in \Ll\,,
\end{equation}
which is more than what we need (namely that it just holds for
$\nu$-almost every $x$). 
We know from \eqref{eq:ChacOrn-1} that the limit exists in terms of
conditional expectations for $\nu$-almost every $x$, so that we only
have to show that that it is $\Prob \otimes \nu$-almost everywhere constant.
\\[8pt]
\emph{Step 1. Independence of $x$.} 
Let $K_0 \subset \Ll$ be compact such that the support of $\varphi$ is 
contained in $K_0$. Define $K = \{ x \in \Ll : d(x,K_0) \le 1 \}$.
Given $\ep > 0$, let $0 < \delta \le 1$ be such that 
$|\varphi(x)-\varphi(y)| < \ep$ whenever $d(x,y) < \delta$.

By \eqref{eq:ChacOrn-1}, there is $x$ such that the limits
$\lim_n S_n^x\uno_K\big/S_n^x\varphi$ 
and 
$Z_{\varphi,\psi}= \lim_n S_n^x\varphi\big/S_n^x\psi$ 
exist and are finite $\Prob$-almost surely. 

Local contractivity implies that for this specific $x$ and each $y \in \Xx$,
we have the following. $\Prob$-almost surely, there is a random $N \in \N$ 
such that 
$$
|\varphi(X_k^x) - \varphi(X_k^y)| \le \ep \cdot\uno_{K}(X_k^x) \quad
\text{for all}\; k \ge N.
$$
Therefore, for every $\ep > 0$ and $y \in \Xx$ 
$$
\limsup_{n \to \infty} \frac{|S_n^x\varphi - S_n^y\varphi|}{S_n^x\varphi} \le
\ep\cdot \lim_{n \to \infty} \frac{S_n^x\uno_K}{S_n^x\varphi}
\quad \Prob\text{-almost surely.}
$$
This yields that for every $y \in \Ll$,
$$
\lim_{n \to \infty} \frac{S_n^x\varphi - S_n^y\varphi}{S_n^x\varphi} = 0\,, 
\quad\text{that is,}\quad
\lim_{n \to \infty} \frac{S_n^y\varphi}{S_n^x\varphi} = 1 \quad  
\Prob\text{-almost surely.}
$$
The same applies to $\psi$ in the place of $\varphi$. We get that for all $y$,
$$
\frac{S_n^x\varphi}{S_n^x\psi}- \frac{S_n^y\varphi}{S_n^y\psi} = 
\frac{S_n^y\varphi}{S_n^y\psi}\left(\frac{S_n^x\varphi}{S_n^y\varphi}
\frac{S_n^y\psi}{S_n^x\psi} -1\right)
\to 0\quad\Prob\text{-almost surely.}
$$
In other terms, for the positive random variable $Z_{\varphi,\psi}$ given above
in terms of our $x$, 
$$
\lim_{n \to \infty}\frac{S_n^y\varphi}{S_n^y\psi} = Z_{\varphi,\psi} \quad 
\text{$\Prob$-almost surely, for every}\;y \in \Ll\,.
$$
\\[4pt]
\emph{Step 2. $Z_{\varphi,\psi}$ is a.s. constant.} Recall the random variables 
$X_{m,n}^x$ of \eqref{eq:Xmn} and set $S_{m,n}^x\varphi(\omega) 
= \sum_{k=m}^n \varphi\bigl(X_{m,k}^x(\omega)\bigr)$, $n > m$.
Then Step 1 also yields that for our given $x$ and each $m$,
\begin{equation}\label{eq:limit}
\lim_{n \to \infty}\frac{S_{m,n}^y\varphi}{S_{m,n}^y\psi} 
= \lim_{n \to \infty}\frac{S_{m,n}^x\varphi}{S_{m,n}^x\psi}
\quad  \text{$\Prob$-almost surely, for every}\;x \in \Ll\,.
\end{equation}
Let $\Omega_0 \subset \Omega$ 
be the set on which the convergence in \eqref{eq:limit} 
holds for all $m$, and both $S_n^x\varphi$ and $S_n^x\psi \to \infty$ on 
$\Omega_0\,$. We have $\Prob(\Omega_0)=1$. 
%On $\Omega_0$, also
%$$
%\lim_{n \to \infty}\frac{S_{m,n}^x\varphi}{S_{m,n}^x\psi} = Z_{\varphi,\psi} \quad 
%\text{for every} \;x \in \Ll\,.
%$$
%In particular, $\Omega_0$ is invariant with respect to the shift of
%the sequence $(F_n)_{n\ge 1}$.
For fixed $\omega \in \Omega_0$ and $m \in \N$, let $y = X_m^x(\omega)$. Then 
(because in the ratio limit we can omit the first $m$ terms of the sums)
$$
Z_{\varphi,\psi}(\omega)= 
\lim_{n\to \infty}\frac{S_n^x\varphi(\omega)}{S_n^x\psi(\omega)}
=\lim_{n\to \infty}\frac{S_{m,n}^y\varphi(\omega)}{S_{m,n}^y\psi(\omega)} 
=\lim_{n\to \infty}\frac{S_{m,n}^x\varphi(\omega)}{S_{m,n}^x\psi(\omega)}.
$$
Thus, $Z_{\varphi,\psi}$ is independent of $F_1, \dots, F_m\,$, 
whence it is constant by Kolmogorov's 0-1 law. This 
completes the proof of ergodicity. It is immediate from
\eqref{eq:ratio} that $\nu$ is unique up to multiplication by constants. 
\end{proof}

\begin{cor}\label{cor:returntime}
Let the locally contractive SDS $(X_n^x)$ be recurrent with invariant
Radon measure $\nu$. For relatively compact, open $U \subset \Xx$
which intersects $\Ll$, consider the probability measure $\msf_U$ on $\Xx$ 
defined by $\msf_U(B) = \nu(B\cap U)/\nu(U)$. Consider the
SDS with initial distribution $\msf_U$, and let $\tau^U$ be its return
time to $U$.
\\[5pt]
\emph{(a)} If $\nu(\Ll) < \infty$ then the SDS is \emph{positive recurrent},
that is,
$$
\Ex(\tau^U) = \nu(\Ll)/\nu(U) < \infty\,.
$$
\\[5pt] 
\emph{(b)} If $\nu(\Ll) = \infty$ then the SDS is \emph{null recurrent},
that is,
$$
\Ex(\tau^U) = \infty\,.
$$
\end{cor}

This follows from the well known formula of Kac, 
see e.g. {\sc Aaronson}~\cite[1.5.5., page~44]{Aa}.

\begin{lem}\label{lem:posrec} In the positive recurrent case,
let the invariant measure be normalised such that $\nu(\Ll)=1$.
Then, for every starting point $x \in X$, the sequence $(X_n^x)$
converges in law to $\nu$.
\end{lem}

\begin{proof} Let $\varphi:\Xx \to \R$ be continuous and compactly supported.
Since $\varphi$ is uniformly continuous, local contractivity yields 
for all $x,y \in X$ that $\varphi(X_n^x) - \varphi(X_n^y) \to 0$ almost surely. 
By dominated convergence, $\Ex\bigl(\varphi(X_n^x) - \varphi(X_n^y)\bigr) \to 0$.
Thus, 
$$
P^n\varphi(x) - \int \varphi \, d\nu 
= \int \bigl( P^n\varphi(x) - P^n\varphi(y) \bigr)\,d\nu(y) 
= \int \Ex\bigl(\varphi(X_n^x) - \varphi(X_n^y)\bigr)\,d\nu(y) \to 0
$$
\end{proof}

\section{Basic example: the affine stochastic recursion}\label{sec:affine}

Here we briefly review the main known results regarding the SDS on $\Xx=\R$ given by 
\begin{equation}\label{eq:affine}
Y_0^x = x\,,\quad Y_{n+1}^x = A_nY_n^x + B_{n+1}\,,
\end{equation}
where $(A_n,B_n)_{n \ge 0}$ is a sequence of i.i.d. random variables
in $\R^+_* \times \R$. The following results are known. 

\begin{pro}\label{pro:contract}
If $\Ex(\log^+ A_n) < \infty$ and 
$$
-\infty \le \Ex (\log A_n) < 0
$$
then $(Y_n^x)$ is strongly contractive on $\R$. 

If in addition
$\Ex(\log^+ |B_n|) < \infty$ then the affine SDS has a unique invariant
probability measure $\nu$, and is (positive) recurrent on $\Ll = \supp(\nu)$.
Furthermore, the shift on the trajectory space is ergodic with respect
to the probability measure $\Pr_{\nu}\,$.
\end{pro}

\begin{proof}[Proof (outline)]
This is \emph{the} classical application of Furstenberg's contraction
principle. One verifies that for the associated right process,
$$
R_n^x \to Z = \sum_{n=1}^{\infty} A_1 \cdots A_{n-1}B_n
$$
almost surely for every $x \in \R$. The series that defines $Z$
is almost surely abolutely convergent by the assumptions on
the two expectations. 
Recurrence is easily deduced via Lemma \ref{lem:transient}. Indeed, we cannot
have $|Y_n^x| \to \infty$ almost surely, because then by dominated convergence
$\nu(U) = \nu \,P^n(U) \to 0$ for every relatively compact set $U$.
Ergodicity now follows from strong contractivity.
\end{proof}

%(Recall that in the 
%present paper, ``recurrence'' always refers to topological recurrence.) 

\begin{pro}\label{pro:centered}
Suppose that $\Prob[A_n = 1] < 1$ and $\Prob[A_n x + B_n = x] < 1$
for all $x \in \R$ (non-degeneracy).
If $\Ex(|\log A_n|) < \infty$ and $\Ex(\log^+ B_n) < \infty$, and if
$$
\Ex (\log A_n) = 0
$$
then $(Y_n^x)$ is locally contractive on $\R$. 

If in addition
$\Ex(|\log A_n|^2) < \infty$ and
$\Ex\bigl((\log^+ |B_n|)^{2+\ep}\bigr) < \infty$ for some $\ep > 0$
then the affine SDS has a unique invariant Radon
measure $\nu$ with infinite mass, and it is (null) recurrent on 
$\Ll = \supp(\nu)$.
\end{pro}

This goes back to \cite{BBE}, with a small gap that was later filled in
\cite{Be1}. With the moment conditions as stated here, a nice and complete
``geometric'' proof is given in \cite{Br}: it is shown that under the stated
hypotheses,
$$
A_1 \cdots A_n \cdot \uno_K(Y_n) \to 0 \quad \text{almost surely}
$$
for very compact set $K$. Recurrence was shown earlier in
\cite[Lemma 5.49]{E2}.

\begin{pro}\label{pro:expanding}
If $\Ex(|\log A_n|) < \infty)$ and $\Ex(\log^+ B_n) < \infty$, and if
$$
\Ex (\log A_n) > 0
$$
then $(Y_n^x)$ is transient, that is, $|Y_n^x| \to \infty$
almost surely for every starting point $x \in \R$.
\end{pro}

A proof is given, e.g., by {\sc Elie}~\cite{E3}.

\section{Iteration of random contractions}\label{sec:contract}

Let us now consider a more specific class of SDS:  
within $\Gp$, we consider the closed submonoid $\Lp_1$ of all 
\emph{contractions} of $\Xx$, i.e., mappings $f: \Xx \to \Xx$ with 
Lipschitz constant $\lp(f) \le 1$. We suppose that the probability measure
$\wt\mu$ that governs the SDS is supported by $\Lp_1$, that is, each 
random function $F_n$ of \eqref{eq:SDS} satisfies $\lp(F_n) \le 1$.
In this case, one does not need local contractivity in order to obtain
Lemma \ref{lem:transient}; % and the existence of the attractor $\Ll\,$;
this follows directly from properness of $\Xx$ and the inequality
$$
D_n(x,y) \le d(x,y)\,,\quad\text{where}\quad D_n(x,y) = d(X_n^x,X_n^y)\,.
$$
When $\Prob[d(X_n^x,x) \to \infty]=0$ for every $x$, we
can in general only speak of conservativity, since we do not
yet have an attractor on which the SDS is topologically recurrent.  
Let $\Sf(\wt\mu)$ be the closed sub-semigroup of $\Lp_1$ generated by 
$\supp(\wt\mu)$. 

\begin{rmk}\label{rmk:contr} For strong contractivity it is sufficient
that $\Prob[D_n(x,y)\to 0]=1$ pointwise for all $x,y \in \Xx$.

Indeed, by properness, $\Xx$  has a dense, countable subset $Y$. 
If $K \subset \Xx$ is compact and $\ep > 0$ then there is
a finite $W \subset Y$ such that $d(y, W) < \ep$ for every $y \in K$.
Therefore  
$$
\sup_{y \in K} D_n(x,y) \le 
\underbrace{\max_{w \in W} D_n(x,w)}_{\textstyle{\to 0 \;\text{a.s.}}} 
+ \ep\,,
$$ 
since $D_n(x,y) \le D_n(x,w)+D_n(w,y) \le  D_n(x,w) + d(w,y)$.
\end{rmk}

The following key result of \cite{Be} (whose statement and proof we have
slightly strengthened here) is inspired by \cite[Thm. 2.2]{Kn}, 
where reflected random walk is studied; see also \cite{Le}. 

\begin{thm}\label{thm:contractive} If the SDS of
contractions is conservative, then it is strongly contractive if and only if 
$\Sf(\wt\mu) \subset \Lp_1$ contains  a constant function.
\end{thm}

\begin{proof} Keeping Remark \ref{rmk:contr} in mind, first assume that
$D_n(x,y) \to 0$ almost surely for all $x,y$. 
We can apply all previous results on (local) contractivity,
and the SDS has the non-empty attractor $\Ll$. If $x_0 \in \Ll$, then 
with probability $1$ there is a random subsequence $(n_k)$ such that
$X_{n_k}^x \to x_0$ for every $x \in \Xx$, and by the above, this convergence
is uniform on compact sets. Thus, the constant
mapping $x \mapsto x_0$ is in $\Sf(\wt\mu)$.

\smallskip

Conversely, assume that $\Sf(\wt\mu)$ contains a constant function.
Since $D_{n+1}(x,y) \le D_n(x,y)$, the limit $D_{\infty}(x,y) = \lim_n D_n(x,y)$
exists and is between $0$ and $d(x,y)$. We set 
$w(x,y) = \Ex\bigl( D_{\infty}(x,y)\bigr)$.
First of all, we claim that
\begin{equation}\label{eq:martingale}
\lim_{m \to \infty} w(X_m^x\,,X_m^y) = D_{\infty}(x,y)\quad\text{almost surely.}
\end{equation}
To see this, consider $X_{m,n}^x$ as in \eqref{eq:Xmn}.  
Then $D_{m,\infty}(x,y) = \lim_n d(X_{m,n}^x,X_{m,n}^y)$ has the same
distribution as $D_\infty(x,y)$, whence 
$\Ex\bigl(D_{m,\infty}(x,y)\bigr) = w(x,y)$. Therefore, we also have
$$
\Ex\bigl(D_{m,\infty}(X_m^x\,,X_m^y) \mid F_1, \ldots, F_m\bigr) = 
w(X_m^x\,,X_m^y)\,.
$$
On the other hand, $D_{m,\infty}(X_m^x\,,X_m^y) = D_{\infty}(x,y)$,
and the bounded martingale
$$
\Bigl(\Ex\bigl(D_{\infty}(x,y) \vert F_1, \ldots, F_m \bigr)
       \Bigr)_{m \ge 1}
$$       
converges  almost surely % and  in $L^1$  
to $D_\infty(x, y)$. Statement \eqref{eq:martingale} follows.

\smallskip

Now let $\ep > 0$ be arbitrary, and fix $x, y \in X$.
We have to show that the event $\Lambda = [D_{\infty}(x,y) \ge \ep]$ has probability 
$0$. 

(i) By conservativity, 
$$
\Pr\left( \bigcup_{r \in \N\,} \bigcap_{\,m \in \N\,} \bigcup_{\,n \ge m\,} 
[X_n^x\,,\; X_n^y \in \Bb(r)] \right) = 1\,.
$$
On $A$, we have $D_n(x,y) \ge \ep$ for all $n$. Therefore we need to show
that $\Pr(\Lambda_r) = 0$ for each $r \in \N$, where 
$$
\Lambda_r = \bigcap_{m \in \N} \bigcup_{n \ge m} 
[X_n^x\,,\; X_n^y \in \Bb(r)\,,\;D_n(x,y) \ge \ep]\,.
$$
(ii) By assumption, there is $x_0 \in X$ which can be approximated
uniformly on compact sets by functions of the form $f_k \circ \dots \circ f_1$,
where $f_j \in \supp(\wt\mu)$. Therefore, given $r$ there is $k \in \N$ 
such that
$$
\Pr(\Gamma_{k,r}) > 0\,, \quad \text{where} \quad 
\Gamma_{k,r} = \left[\sup_{u \in \Bb(r)} d(X_k^u\,,x_0) \le \ep/4 \right]\,.
$$
On $\Gamma_{k,r}$ we have $D_{\infty}(u,v) \le D_k(u,v) \le \ep/2$ for all 
$u,v \in \Bb(r)$. Therefore, setting $\de=\Pr(\Gamma_{k,r})\cdot (\ep/2)$,
we have for all $u,v \in \Bb(r)$ with $d(u,v) \ge \ep$ that
$$
\begin{aligned}
w(u,v) &= \Ex\bigl( \uno_{\Gamma_{k,r}}\, D_{\infty}(u,v) \bigr) +
\Ex\bigl( \uno_{\Xx \setminus \Gamma_{k,r}}\, D_{\infty}(u,v) \bigr) \\
&\le \Pr(\Gamma_{k,r})\cdot (\ep/2) + \bigl(1-  \Pr(\Gamma_{k,r})\bigr)\cdot d(u,v)
\le d(u,v) - \de\,.
\end{aligned}
$$
We conclude that on $\Lambda_r$, there is a (random) sequence $(n_{\ell})$ such that 
$$
w(X_{n_{\ell}}^x\,,X_{n_{\ell}}^y) \le D_{n_{\ell}}(x,y) - \de\,.
$$
Passing to the limit on both sides, we see that \eqref{eq:martingale} is
violated on $\Lambda_r$, since $\de > 0$. Therefore $\Pr(\Lambda_r) = 0$ 
for each $r$.
\end{proof}

\begin{cor}\label{cor:loccontractive} 
If the semigroup $\Sf(\wt\mu) \subset \Lp_1$
contains  a constant function, then the SDS is locally contractive.
\end{cor}

\begin{proof} In the transient case, $X_n^x$ can visit any compact
$K$ only finitely often, whence $d(X_n^x, X_n^y)\cdot \uno_K(X_n^x) = 0$ for all
but finitely many $n$.
In the conservative case, we even have strong contractivity by
Proposition \ref{thm:contractive}.
\end{proof}

\section{Some remarks on reflected random walk}\label{sec:reflected}

As outlined in the introduction, the refleced random walk on $\R^+$
induced by a sequence $(B_n)_{n \ge 0}$ of i.i.d. real valued random variables
is given by
\begin{equation}\label{eq:refl}
X_0^x = x \ge 0\,,\quad X_{n+1}^x = |X_n^x - B_{n+1}|\,.
\end{equation}
Let $\mu$ be the distribution of the $B_n\,$, a probability measure on $\R$.
The transition probabilities of reflected random walk are 
$$
P(x,U) = \mu(\{ y : |x-y| \in U \})\,,
$$
where $U \subset \R^+$ is a Borel set.
When $B_n \le 0$ almost surely, then $(X_n^x)$ is an ordinary random
walk (resulting from a sum of i.i.d. random variables). We shall exclude this,
and we shall always assume to be in the \emph{non-lattice} situation. That is, 
\begin{equation}\label{eq:nonlattice} 
\supp(\mu) \cap (0\,,\,\infty) \ne \emptyset\,,\quad 
\text{and there is no}\;\kappa > 0\;\text{such that}\quad\supp(\mu) \subset 
\kappa\cdot\Z\,.
\end{equation}
For the lattice case, see \cite{PeWo}.

For $b \in \R$, consider $g_b \in \Lp_1\bigl(\R^+\bigr)$ given by 
$g_b(x) = |x-b|$.
Then our reflected random walk is the SDS on $\R^+$ induced by the 
random 
continuous contractions $F_n = g_{B_n}\,$, $n \ge 1$. The law $\wt \mu$ of the
$F_n$ is the image of $\mu$ under the mapping $b \mapsto g_b\,$.

In \cite[Prop. 3.2]{Le}, it is shown that $\Sf(\wt\mu)$ contains the constant 
function $x \mapsto 0$. Note that this statement
and its proof in \cite{Le} are completely deterministic, regarding 
topological properties of the set $\supp(\mu)$. In view of  
Theorem \ref{thm:contractive} and Corollary \ref{cor:loccontractive},
we get the following.

\begin{pro}\label{pro:reflcont}
Under the assumptions \eqref{eq:nonlattice}, reflected random walk on
$\R^+$ is locally contractive, and strongly contractive if it is recurrent.
\end{pro}

\subsection*{A. Non-negative $B_n\,$}$\,$
\\[5pt]
We first consider the case when $\Prob[B_n \ge 0] = 1$. Let
$$
N = \sup \supp(\mu) \AND 
\Ll = \begin{cases} [0\,,\,N], &\text{if}\; N < \infty\,,\\
\R^+, &\text{if}\; N = \infty\,.
\end{cases}                      
$$               
The distribution function of $\mu$ is 
$$
F_{\mu}(x) = \Prob[B_n \le x] = \mu\bigl([0\,,\,x]\bigr),\; x \ge 0\,.
$$ 
We next subsume basic properties that are due to \cite{Fe}, \cite{Kn} and 
\cite{Le}; they do not depend on recurrence. 

\begin{lem}\label{lem:irreducible} Suppose that \eqref{eq:nonlattice} is
verified and that $\supp(\mu) \subset \R^+$. 
Then the following holds.\\[4pt]  
\emph{(a)} 
The reflected random walk with any starting
point is absorbed after finitely many steps by the interval $\Ll$.\\[4pt] 
{\rm (b)} It is topologically irreducible on $\Ll$, that is, 
for every $x \in \Ll$ and open set $U \subset \Ll$, there is $n$ such that 
$P^n(x,U) = \Prob[X_n^x \in U] > 0\,.$\\[4pt] 
{\rm (c)} The measure $\nu$ on $\Ll$ given by
$$
\nu(dx) = \bigl(1-F_{\mu}(x)\bigr)\,dx\,,
$$
where $dx$ is Lebesgue measure, is an invariant measure for the transition
kernel $P$. 
\end{lem}

At this point Lemma \ref{lem:exc-inv} implies that in the recurrent case,
the above set is indeed the attractor, and $\nu$ is the unique invariant 
measure up to multiplication with constants.
We now want to understand when we have recurrence.

\begin{thm}\label{thm:refl-recurr} Suppose that \eqref{eq:nonlattice} is
verified and that $\supp(\mu) \subset \R^+$.
Then each of the following conditions implies the next one and is sufficient
for recurrence of the reflected random walk on $\Ll$.
\begin{gather}
\Ex(B_1) < \infty \tag{i}\label{cond1}\\[3pt]
\Ex\bigl(\sqrt{B_1}\,\bigr) < \infty \tag{ii}\label{cond2}\\
\int_{\R^+} \bigl(1-F_{\mu}(x)\bigr)^2\,dx < \infty
\tag{iii}\label{cond3}\\
\lim_{y \to \infty} \bigl(1-F_{\mu}(y)\bigr) 
\int_0^y \bigl(F_{\mu}(y)-F_{\mu}(x)\bigr)\,dx = 0 \tag{iv}\label{cond4}
\end{gather}
In particular, one has positive recurrence precisely when $\Ex(B_1) < \infty$.
\end{thm}

The proof of \eqref{cond1} $\implies$ \eqref{cond2} $\implies$ 
\eqref{cond3} $\implies$ \eqref{cond4} is a basic
exercise. For condition \eqref{cond1}, see \cite{Kn}. The implication
\eqref{cond2} $\implies$ recurrence is due to \cite{Sm}, while the recurrence
condition  \eqref{cond3} was proved by ourselves in \cite{PeWo}. However,
we had not been aware of \cite{Sm}, as well as of \cite{Rab}, where
it is proved that already \eqref{cond4} implies recurrence on $\Ll$.
Since $\nu$ has finite total mass precisely when $\Ex(B_1) < \infty$,
the statement on positive recurrence follows from Corollary 
\ref{cor:returntime}. In this case, also Lemma \ref{lem:posrec} applies and
yields that $X_n^x$ converges in law to $\frac{1}{\nu(\Ll)}\nu$. This
was already obtained by \cite{Kn}.

Note that the ``margin'' between conditions \eqref{cond2}, \eqref{cond3}
and \eqref{cond4} is quite narrow. 

\subsection*{B. General reflected random walk}$\,$
\\[5pt]
We now drop the restriction that the random variables $B_n$
are non-negative. Thus, the ``ordinary'' random walk 
$S_n = B_1 + \cdots + B_n$ on $\R$ may visit the positive
as well as the negative half-axis. 
Since we assume that $\mu$ is non-lattice, the closed group
generated by $\supp(\mu)$ is $\R$. 

We start with a simple observation (\cite{Be2} has a more complicated proof).

\begin{lem}\label{lem:symmetric}
If $\mu$ is symmetric, then reflected random walk is (topologically) recurrent 
if and only if the random walk $(S_n)$ is recurrent.
\end{lem}

\begin{proof} If $\mu$ is symmetric, then also $|S_n|$ is a Markov chain.
Indeed, for a Borel set $U \subset \R^+$, 
$$
\begin{aligned}
\Prob[\,|S_{n+1}| \in U \mid S_n=x] 
&= \mu(-x+U) + \mu(-x-U) - \mu(-x)\, \de_0(U)\\
&=\Prob[\,|S_{n+1}| \in U \mid S_n=-x]\,,
\end{aligned}
$$ 
and we see that $|S_n|$ has the same transition probabilities as the 
reflected random walk governed by $\mu$. 
\end{proof}   

Recall the classical result that when $\Ex(|B_1|) <\infty$ and
$\Ex(B_1)=0$ then $(S_n)$ is recurrent; see {\sc Chung and Fuchs~\cite{ChFu}}. 
\begin{cor}\label{cor:symm}
If $\mu$ is 
symmetric and has finite first moment then reflected random walk is recurrent. 
\end{cor}

Let $B_n^+ = \max \{B_n, 0 \}$ and $B_n^- = \max \{-B_n, 0\}$, so that
$B_n = B_n^+ - B_n^-$. The following is well-known.
\begin{lem}\label{lem:infty} 
If \emph{(a)} $\;\Ex(B_1^-) < \Ex(B_1^+) \le \infty\,$, or if
\emph{(b)} $\;0 < \Ex(B_1^-) = \Ex(B_1^+) < \infty\,$, then
$\limsup S_n = \infty\,$ almost surely, so that there are infinitely many
reflections.
\end{lem}
In general, we should exclude that $S_n \to -\infty$, since in that 
case there are only finitely many
reflections, and reflected random walk tends to $+\infty$ almost surely.
In the sequel, we assume that $\limsup S_n = \infty$ almost surely.
Then the (non-strictly) ascending  \emph{ladder epochs}
$$
\lb(0)  = 0\,,\quad  \lb(k+1) = \inf \{ n > \lb(k) : S_n \ge S_{\lb(k)} \}
$$ 
are all almost surely finite, and the random variables $\lb(k+1) - \lb(k)$ 
are i.i.d.
We can consider the \emph{embedded random walk} $S_{\lb(k)}\,$, $k \ge 0$, 
which tends to $\infty$ almost surely. Its increments 
$\overline B_k = S_{\lb(k)} - S_{\lb(k-1)}\,$, $k \ge 1$, are i.i.d. 
non-negative random variables with distribution denoted $\overline{\mu}$. 
Furthermore, if $\overline{\!X}_k^x$ denotes the reflected random
walk associated with the sequence $(\overline B_k)$, while $X_n^x$ is our original
reflected random walk associated with $(B_n)$, then 
$$
\overline{\!X}_k^x = X_{\lb(k)}^x\,,
$$ 
since no reflection can occur between times $\lb(k)$ and $\lb(k+1)$.
When $\Prob[B_n < 0] > 0$, one clearly has $\sup \supp(\overline{\mu}) 
= +\infty\,$.
Lemma \ref{lem:irreducible} implies the following.

\begin{cor}\label{cor:irreducible}
Suppose that \eqref{eq:nonlattice} is verified, $\Prob[B_n < 0] > 0$ and 
$\limsup S_n = \infty$. Then\\[4pt]  
\emph{(a)}  reflected random walk is topologically 
irreducible on $\Ll = \R^+$, and
\\[4pt]  
\emph{(b)} 
the embedded reflected random walk $\overline{\!X}_k^x$ is recurrent if and 
only the original reflected random walk is recurrent.
\end{cor} 

\begin{proof} Statement (a) is clear.  

Since both processes are locally contractive, each of the
two processes is transient if and only if it tends to $+\infty$ almost surely:
If $\lim_n X_n^x = \infty$ then clearly also $\lim_k X_{\lb(k)}^x = \infty$ a.s.
Conversely, suppose that $\lim_k \overline{\!X}_k^x \to \infty$ a.s.
If $\lb(k) \le n < \lb(k+1)$ then $X_n^x \ge X_{\lb(k)}^x$.
(Here, $k$ is random, depending on $n$ and $\omega \in \Omega$, and when
$n \to \infty$ then $k \to \infty$ a.s.) Therefore, also 
$\lim_n X_n^x = \infty$ a.s., so that (b) is also true.
\end{proof}

We can now deduce the following.

\begin{thm}\label{thm:sqrt} Suppose that \eqref{eq:nonlattice} is verified and 
that $\Prob[B_1 < 0]>0$. 
Then reflected random walk  $(X_n^x)$ is (topologically) recurrent on 
$\Ll = \R^+$, if
\begin{quote} 
\emph{(a)} $\;\Ex(B_1^-) < \Ex(B_1^+)$ and 
$\Ex\bigl(\sqrt{B_1^+}\,\bigr) < \infty\,,$ or if\\
\emph{(b)} $\;0 < \Ex(B_1^-) = \Ex(B_1^+)$ and 
$\Ex\Bigl(\sqrt{B_1^+}^{\,3}\Bigr) < \infty\,$. 
\end{quote}
\end{thm}

\begin{proof}
We show that in each case the assumptions imply that 
$\Ex\bigl(\sqrt{\,\overline{B}_1}\bigr) < \infty$.
Then we can apply Theorem \ref{thm:refl-recurr} to deduce recurrence of 
$(\overline{\!X}_k^x)$.
This in turn yields recurrence of $(X_n^x)$ by Corollary \ref{cor:irreducible}.

\smallskip

(a) Under the first set of assumptions,
$$
\begin{aligned}
\Ex\Bigl(\sqrt{\overline{B}_1}\Bigr) 
&= \Ex\Bigl(\sqrt{B_1+\ldots+B_{\lb(1)}^{\,} }\,\Bigr)
\le \Ex\Bigl(\sqrt{B_1^+ +\ldots+B_{\lb(1)}^+}\,\Bigr)\\
&\le \Ex\Bigl(\sqrt{B_1^+}+\ldots+\sqrt{B_{\lb(1)}^+}\,\Bigr) 
= \Ex\Bigl(\sqrt{B_1^+}\,\Bigr) \cdot\Ex\bigl(\lb(1)\bigr) 
\end{aligned}
$$
by Wald's identity. Thus, we now are left with proving
$\Ex\bigl(\lb(1)\bigr) < \infty\,$. If $\Ex(B_1^+) < \infty$, then
$\Ex(|B_1|) < \infty$ and $\Ex(B_1) > 0$ by assumption,
and in this case it is well known that $\Ex\bigl(\lb(1)\bigr) < \infty\,$;
see e.g. \cite[Thm. 2 in \S XII.2, p. 396-397]{Fe}. 
If  $\Ex(B_1^+) = \infty$ then there is $M > 0$ such that
$B_n^{(M)} = \min\{ B_n\,, M\}$ (which has finite first moment)
satisfies $\Ex(B_n^{(M)}) = \Ex(B_1^{(M)})> 0\,$. The first increasing
ladder epoch $\lb^{(M)}(1)$ associated with 
$S_n^{(M)} = B_1^{(M)} + \ldots + B_n^{(M)}$ has finite expectation by
what we just said, and $\lb(1) \le\lb^{(M)}(1)$. Thus, $\lb(1)$
is integrable.

\smallskip

(b) If the $B_n$ are centered, non-zero and 
$\Ex\bigl((B_1^+)^{1+a}\bigr)< \infty\,,$ where $a > 0$, then 
$\Ex\bigl((\overline{B}_1)^a \bigr) < \infty\,$, as was shown by 
{\sc Chow and Lai~\cite{ChLa}}. In our case, $a=1/2$.
\end{proof}

We conclude our remarks on reflected random walk by discussing sharpness of 
the sufficient recurrence conditions
$\Ex\Bigl(\sqrt{B_1^+}^{\,3}\Bigr) < \infty$ in the centered case,
resp. $\Ex\bigl(\sqrt{B_1}\bigr) < \infty$ in the case when $B_1 \ge 0$.

\begin{exa}\label{ex:symm}
Define a symmetric probability measure $\mu$ on $\R$ by 
$$
\mu(dx)=\frac{dx}{(1+|x|)^{1+a}}\,,
$$
where $a >0$ and $c$ is the proper normalizing constant (and $dx$ is
Lebesgue measure). Then it is well known and quite easy to prove via Fourier
analysis that the associated symmetric random walk $S_n$ on $\R$ 
is recurrent if and only if $a \ge 1$. 
By Lemma \ref{lem:symmetric}, the associated 
reflected random walk is also recurrent, but when $1 \le a \le 3/2$ then
condition (b) of Theorem \ref{thm:sqrt} does not hold. 
\end{exa} 

Nevertheless, we can also show that in general, the sufficient condition 
$\Ex\Bigl(\sqrt{\,\overline{B}_1\,}\Bigr) < \infty$ for
recurrence of reflected random walk with non-negative increments 
$\overline{B}_n$ is very close to being sharp. (We write $\overline{B}_n$
because we shall represent this as an embedded random walk in the next 
example.)

\begin{pro}\label{pro:nsymm}
Let $\mu_0$ be a probability measure on $\R^+$ which has a density 
$\phi_0(x)$
with respect to Lebesgue measure that is decreasing and satisfies
$$
\phi(x) \sim c \,(\log x)^b \big/ x^{3/2}\,,
\quad\text{as}\;x\to \infty\,,
$$
where $b > 1/2$ and $c > 0$.
Then the associated reflected random walk on $\R^+$ is transient.
\end{pro}

Note that $\mu_0$  has finite moment of order $\frac12 - \ep$ for 
every $\ep > 0$, while the moment of order $\frac12$ is infinite. 

\smallskip

The proof needs some preparation. Let $(B_n)$ be i.i.d. random variables
with values in $\R$ that have finite first moment and are non-constant
and centered, and let $\mu$ be their common distribution. 

The first \emph{strictly ascending} and 
\emph{strictly descending ladder epochs} of the random walk
$S_n=B_1 + \ldots + B_n$ are 
$$
\tb_+(1)  = \inf \{ n > 0 : S_n > 0 \} \AND
\tb_-(1)  = \inf \{ n > 0 : S_n < 0 \}\,,
$$
respectively. They are almost surely finite. Let $\mu_+$ be the distribution
of $S_{\tb_+(1)}$ and $\mu_-$ the distribution of $S_{\tb_-(1)}$, and --
as above -- $\overline{\mu}$ the distribution of 
$\overline{B}_1 = S_{\lb(1)}\,$.
We denote the characteristic function associated with any probability measure
$\sigma$ on $\R$ by $\wh\sigma(t)\,$, $t \in \R$. Then, following 
{\sc Feller~\cite[(3.11) in \S XII.3]{Fe}}, \emph{Wiener-Hopf-factorization}
tells us that
$$
\begin{gathered}
\mu = \overline{\mu} + \mu_- - \overline{\mu}*\mu_-
\;\AND\;
\overline{\mu} = u\cdot \de_0 + (1-u)\cdot\mu_+\ ,\\
\text{where}\quad
u = \overline{\mu}(0) = 
\sum_{n=1}^{\infty} \Prob[S_1 < 0\,, \ldots, S_{n-1} < 0\,,\; S_n =0] < 1\,. 
\end{gathered}
$$
Here $*$ is convolution. Note that when $\mu$ is absolutely continuous
(i.e., absolutely continuous with respect to Lebesgue measure) then $u = 0$,
so that 
\begin{equation}\label{eq:WH1}
\overline{\mu} = \mu_+ \AND \mu =  \mu_+ + \mu_- - \mu_+*\mu_-\,.
\end{equation}

\begin{lem}\label{lem:ladder}
Let $\mu_0$ be a probability measure on $\R^+$ which has a decreasing density
$\phi_0(x)$ with respect to Lebesgue measure.
Then there is an absolutely continuous symmetric probability measure 
$\mu$ on $\R$ such that that the associated first (non-strictly) ascending 
ladder random variable has distribution $\mu_0$.
\end{lem}

\begin{proof}
If $\mu_0$ is the law of the first strictly ascending  ladder random 
variable associated with some absolutely continuous, symmetric
measure $\mu$, then by \eqref{eq:WH1} we must have $\mu_+ = \mu_0$ and 
$\mu_- = \check\mu_0\,$, the reflection of $\mu_0$ at $0$, and
\begin{equation}\label{eq:mudef}
\mu = \mu_0 + \check \mu_0  - \mu_0*\check\mu_0\,.
\end{equation}
We \emph{define} $\mu$ in this way.
The monotonicity assumption on $\mu_0$ implies that $\mu$ is a probability 
measure: indeed, by the monotonicity assumption it is straightforward to 
check that the function
$\phi = \phi_0 + \check \phi_0 - \phi_0 * \check \phi_0$ 
is non-negative; this is the density of $\mu$.

The measure $\mu$ of \eqref{eq:mudef} is non-degenerate and symmetric. 
If it induces 
a recurrent random walk $(S_n)$, then the ascending and descending ladder
epochs are a.s. finite. If $(S_n)$ is transient, then $|S_n| \to \infty$ 
almost surely, but it cannot be $\Prob[S_n \to \infty] > 0$ since
in that case this probaility had to be 1 by Kolmogorov's 0-1-law, while 
symmetry would yield 
$\Prob[S_n \to -\infty] = \Prob[S_n \to \infty] \le 1/2$. 
Therefore $\liminf S_n = -\infty$ and $\limsup S_n = +\infty$ almost
surely, a well-known fact, see e.g. \cite[Thm. 1 in \S XII.2, p. 395]{Fe}. 
Consequently, the ascending and descending ladder 
epochs are again a.s. finite. Therefore the probability measures 
$\mu_+$ and $\mu_-=\check\mu_+$ (the laws of $S_{\tb_\pm(1)}$) are well defined.
By the uniqueness theorem of Wiener-Hopf-factorization
\cite[Thm. 1 in \S XII.3, p. 401]{Fe}, it follows that 
$\mu_- = \check\mu_0$
and that the distribution of the first (non-strictly) ascending ladder random 
variable is $\overline{\mu} =\mu_0\,$.
\end{proof}
 
\begin{proof}[Proof of Proposition \ref{pro:nsymm}]  
Let $\mu$ be the symmetric measure associated with $\mu_0$
according to \eqref{eq:mudef} in Lemma \ref{lem:ladder}. 
Then its characteristic function
$\wh\mu(t)$ is non-negative real. A well-known
criterion says that the random walk $S_n$ associated with $\mu$ is transient
if and only if (the real part of) $1\big/\bigl(1-\wh\mu(t)\bigr)$ is integrable
in a neighbourhood of $0$. 
Returning to $\mu_0=\mu_+\,$, it is a standard exercise 
(see  \cite[Ex. 12 in Ch. XVII, Section 12]{Fe})
to show that there is $A \in \C\,$, $A \ne 0$
such that its characteristic function satisfies 
$$
\wh{\mu_0}(t)= 1+A\,\sqrt{t}\,(\log t)^b\,\bigl(1+o(t)\bigr)
\quad \text{as} \; t \to 0\,.
$$
By \eqref{eq:WH1},
$$
1 - \wh\mu(t) = 
\bigl(1 - \wh{\mu_+}(t)\bigr)\bigl(1 - \wh{\mu_-}(t)\bigr)\,. 
$$
We deduce    
$$
\wh{\mu}(t)=  1- |A|^2 |t|(\log |t|)^{2b}\,
%
%1 + |A|^2 \, t \,(\log t)^{2b}\,
%
\bigl(1+o(t)\bigr)
\quad \text{as} \; t \to 0\,.
$$
The function $1\big/\bigl(1-\wh\mu(t)\bigr)$ is integrable near $0$. 
By Lemma \ref{lem:symmetric}, 
the associated reflected random walk is transient. But then also 
the embedded reflected random walk associated with $S_{\lb(n)}$ is
transient by Corollary \ref{cor:irreducible}. This is the reflected 
random walk governed by~$\mu_0\,$.
\end{proof}

%%%%%%%%%%%%%%%%%%%%%%%%%%%%%%%%%%%%%%%%%%%%%
\newpage

\addcontentsline{toc}{section}{\\[-8pt] \hspace*{.5cm}PART II. Stochastic dynamical systems 
induced by Lipschitz mappings}
\begin{center}
{\large \bf PART II. Stochastic dynamical systems 
induced by Lipschitz mappings}
\end{center}

\section{The contractive case, and recurrence in the log-centered case}
\label{sec:lip}

We now consider the situation when the i.i.d. random mappings
$F_n:\Xx \to \Xx$ belong to the semigroup $\Lp \subset \Gp$ of
Lipschitz mappings. Recall our notation $\lp(f)$ for the Lipschitz constant
of $f \in \Lp$. 
%%%We exclude the trivial case where $F_n$ can be constant, that is, 
We assume that
\begin{equation}\label{eq:assume1}
\Prob[\lp(F_n) > 0] = 1\AND \Prob[\lp(F_n) < 1] > 0\,.
%,\;\ \text{and}\;\ \Prob[F_n(x) = x] < 1\; \text{for every}\; x \in \Xx\,.
\end{equation} 
In this situation, the real random variables 
\begin{equation}\label{eq:lip-an-bn}
A_n = \lp(F_n) \AND B_n = d\bigl(F_n(o),o\bigr)
\end{equation}
play an important role. Indeed, let $(X_n^x)$ be the SDS starting at
$x \in \Xx$ which is associated with the sequence $(F_n)$, and for any
starting point $y \ge 0$,
let $(Y_n^y)$ the affine SDS on $\R^+$ associated with $(A_n,B_n)$
according to \eqref{eq:affine}. Then 
\begin{equation}\label{eq:dominate}
d(X_n^x,o) \le Y_n^{|x|}\,,\quad\text{where }\; |x| = d(x,o).
\end{equation}
Thus, we can use the results of Section \ref{sec:affine}. 
First of all, Propositions \ref{pro:furst}, resp. \ref{pro:contract} 
yield the following.

\begin{cor}\label{cor:contract} Given the random i.i.d. Lipschitz mappings
$F_n\,$, let $A_n$ and $B_n$ be as in 
\eqref{eq:lip-an-bn}. 

If $\Ex(\log^+ A_n) < \infty$ and 
$-\infty \le \Ex (\log A_n) < 0$
then the SDS $(X_n^x)$ generated by the $F_n$ is strongly contractive on $\Xx$. 

If in addition
$\Ex(\log^+ B_n) < \infty$ then the SDS has a unique 
invariant probability measure $\nu$ on $\Xx$,  it is (positive) recurrent on 
$\Ll = \supp(\nu)$, and the time shift on the tracetory space $\Xx^{\N_0}$
is ergodic with respect to the probability measure $\Prob_{\nu}\,$.
\end{cor}

\begin{proof} Strong contractivity is obvious. When
$\Ex(\log^+ B_n) < \infty$, \eqref{eq:dominate} tells us that along
with $(Y_n^{|x|})$ also $(X_n^x)$ is positive recurrent. 
\end{proof}

The interesting and much harder case is the one where 
%$A_n = \lp(F_n)$ is such that 
$\log A_n$ is integrable and centered, that is, 
%\begin{equation}\label{eq:cent}
$\Ex(\log A_n) = 0$. The assumptions of Proposition \ref{pro:contract},
applied to $A_n$ and $B_n$ of \eqref{eq:lip-an-bn}, will in general not
imply that our SDS is locally contractive.

\begin{rmks}\label{rmk:hypotheses}
(a) In the log-centered case, we can apply Proposition \ref{pro:centered} to 
$(Y_n^{|x|})$. Among its hypotheses, also need that
\begin{equation}\label{eq:nofix}
\Prob[A_ny + B_n = y] <1 \quad\text{for all}\; y \in \R\,.
\end{equation} 
A sufficient condition for this is that
$$
\Prob[F_n(x) = x] < 1\; \text{for every}\; x \in \Xx\,.
$$
Indeed, when $y=0$, then $\Prob[A_ny + B_n = y] <1$ is the same as 
$\Prob\bigl(F_n(o)=o\bigr) < 1$ from \eqref{eq:assume1}.
If $y \ne 0$ then observe that $A_n-1$ assumes both positive and negative
values with positive probability, so that the requirement is again met.

When the assumptions of  Proposition \ref{pro:centered} hold for the random
variables $(A_n,B_n)$ of \eqref{eq:lip-an-bn}, 
the affine SDS $(Y_n^{|x|})$ on $\R$ is locally contractive and recurrent on its
limit set $\Ll_{\R}\,$, which is contained in $\R^+$ by construction.
Note that it depends on the reference point $o \in \Xx$ through the
definition of $B_n\,$.
\\[5pt]
(b) In view of our assumptions \eqref{eq:assume1}, 
we can always modify the measure 
$\wt \mu$ on $\Lp$ to obtain a new one, say $\wt\mu'$, which has the same support
and satisfies
$$
- \infty < \int_{\Lp} \log \lp(f)\, d\wt\mu'(f) < 0\,.
$$
%For example, we may take $d\wt\mu'(f) = g(f) \,d\wt\mu(f)$, where $g(f) = 
%c \cdot \bigl(\frac14 \uno_{[\lp(f) \ge 1]}(f) + 
%\frac34\uno_{[\lp(f) < 1]}(f)\bigr)$
%with the appropriate normalizing constant $c$. 
Then $\wt\mu'$ gives rise to a strongly contractive SDS. Let $\Ll$ be its limit
set. Remark \ref{rem:non-contractive} tells us that also our original SDS
governed by $\wt\mu$ is topologically irreducible on $\Ll$ and that it 
evolves within $\Ll$ when started in a point of $\Ll$. This set is given by
Corollary \ref{cor:support}. We may assume that the reference point $o$ 
belongs to~$\Ll$. 
\end{rmks}

In the sequel, we shall write
$$
A_{m,m} = 1 \AND A_{m,n} = A_{m+1} \cdots  A_{n-1}A_n \quad (n > m)\,. 
$$

\begin{thm}\label{thm:recur} If in addition to \eqref{eq:assume1} and 
\eqref{eq:nofix}, one has
\begin{equation}\label{eq:assume2}
\Ex (\log A_n) = 0\,,\quad \Ex(|\log A_n|^2) < \infty\,,\AND
\Ex\bigl((\log^+  |B_n|)^{2+\ep}\bigr) < \infty
\end{equation} 
for some $\ep > 0$, then the SDS is topologically recurrent on the set $\Ll$ of 
Corollary \ref{cor:support}. Moreover,
for every $x \in \Xx$ (and not just $\in \Ll$) and every open set $U \subset \Xx$
that intersects $\Ll$,
$$
\Prob[X_n^x \in U \;\text{for infinitely many}\; n] = 1.
$$
\end{thm} 

\begin{proof}
The (non-strictly) \emph{descending} ladder epochs are
$$
\ld(0)  = 0\,,\quad  \ld(k+1) = \inf \{ n > \ld(k) : A_{0,n} \le  A_{0, \ld(k)} \}
$$ 
Since $(A_{0,n})$ is a recurrent multiplicative random walk on $\R^+_*$, these
epochs are stopping times with i.i.d. increments.  
The induced SDS is $(\bar X_k^x)_{k \ge 0}\,$, where 
$\bar X_k^x = X_{\ld(k)}^x\,$. It is also generated by random
i.i.d. Lipschitz mappings, namely 
$$
\bar F_k = F_{\ld(k)} \circ  F_{\ld(k)-1} \circ \dots \circ F_{\ld(k-1)+1}\,,
\quad k \ge 1\,.
$$
With the same stopping times, we also consider the induced affine recursion
given by $\bar Y_k^{|x|} = Y_{\ld(k)}^{|x|}\,$. It is generated by the i.i.d.
pairs  $(\bar A_k,\bar B_k)_{k \ge 1}\,$, where 
$$
\bar A_k = A_{\ld(k-1),\, \ld(k)} \AND 
\bar B_k = \sum_{j=\ld(k-1)+1}^{\ld(k)} |B|_j\, A_{j,\,\ld(k)}\,.
$$
It is known \cite[Lemma 5.49]{E2} that under our assumptions, 
$\Ex(\log^+ \bar A_k) < \infty$, $\Ex(\log \bar A_k) < 0$ and 
$\Ex(\log^+  \bar B_k) < \infty$.
%, so that Proposition  \ref{pro:contract}
%applies to $(\bar Y_k^{|x|})$. 
Returning to $(\bar X_k^x)$, we have
$\lp(\bar F_k) \le \bar A_k$ and $d\bigl( \bar F_k(o),o\bigr) \le \bar B_k$.
Corollary \ref{cor:contract} applies, and the induced SDS is strongly contractive.
It has a unique invariant probability measure $\bar \nu$, and it is (positive)
recurrent on $\bar{\,\Ll} = \supp(\bar \nu)$. Moreover, for every starting point 
$x \in \Xx$ and each open set $U \subset \Xx$ that intersects $\bar{\,\Ll}$,
we get that almost surely, $(\bar X_k^x)$ visits $U$ infinitely often.

In view of the fact that the original SDS is topologically irreducible on
$\Ll$, we have $\bar{\,\Ll} \subset \Ll$. We now define  a sequence
of subsets of $\Ll$ by 
$$
\Ll_0 = \bar{\,\Ll} \AND 
\Ll_m = %\Bigl(
{\textstyle\bigcup} \{ f(\Ll_{m-1}) : f \in \supp(\wt\mu) \}.%\Bigr)^-.
$$   
Then %$\Ll_n \subset \Ll$ for all $n$. But then 
the closure of
$\bigcup_m \Ll_m$ is a subset of $\Ll$ that is mapped into itself by every
$f \in \supp(\wt \mu)$. Corollary \ref{cor:support} yields that 
$$
\Ll = \bigl({\textstyle\bigcup_m} \Ll_m \bigr)^{-}.
$$  
We now show by induction on $m$ that for every starting point $x \in \Xx$ and
every open set $U$ that intersects $\Ll_m\,$,
$$
\Prob[X_n^x \in U \;\text{for infinitely many}\; n] = 1,
$$ 
and this will conclude the proof.

For $m=0$, the statement is true.  Suppose it is true for $m-1$.
Given an open set $U$ that intersects $\Ll_m$, we can find an open,
relatively compact set $V$ that intersects $\Ll_{m-1}$ such that
$\wt\mu(\{ f \in \Lp : f(V) \subset U \} = \alpha > 0$.

By the induction hypothesis, $(X_n^x)$ visits $U$ infinitely often with
probability $1$.
We can now apply Lemma \ref{lem:repeat} with $\ell=2$, $U_0=U$ and
$U_1=V$ to conclude that also $V$ is visited infinitely often with
probability $1$.
\end{proof} 

\begin{lem}\label{lem:inv-exist} \emph{(a)} Under the assumptions 
\eqref{eq:assume1}, every invariant Radon measure $\nu$ satisfies 
$\Ll\subset \supp(\nu)$.
\\[5pt]
\emph{(b)} If in addition to \eqref{eq:assume1}, one has \eqref{eq:nofix} and
\eqref{eq:assume2}, then the SDS possesses an invariant Radon measure
$\nu$ with $\supp(\nu) = \Ll$. Furthermore, the transition operator $P$
is a conservative contraction of $L^1(\Xx,\nu)$ for every invariant measure
$\nu$.
\end{lem}

\begin{proof} (a) Let $\nu$ be invariant.
The argument at the end of the proof of Lemma 
\ref{lem:exc-inv} shows that $f\bigl(\supp(\nu)\bigr) \subset \supp(\nu)$
for all $f \in \supp(\wt \mu)$. As explained in Remark \ref{rmk:hypotheses}(b),
Corollary \ref{cor:support} applies here and yields statement (a).
\\[5pt]
(b) Theorem \ref{thm:recur} yields conservativity. Indeed, let
$\Bb(r)$ be a ball that intersects $\Ll$.  For every starting point 
$x \in \Xx$, the SDS $(X_n^x)$ visits $\Bb(r)$ infinitely often with 
probability $1$. 
We can choose $\varphi \in \Ccal_c^+(\Xx)$ such that $\varphi \ge 1$ on $\Bb(r)$. 
Then 
$$
\sum_{k=1}^{\infty} P^k\varphi(x) = \infty \quad \text{for every}\; x\in \Xx, 
$$
The existence of an invariant Radon measure follows once more from 
\cite[Thm. 5.1]{Li}, and conservativity of $P$ on $L^1(\Xx,\nu)$ follows, see e.g.
\cite[Thm. 5.3]{Rev}.
If right from the start we consider the whole process only on $\Ll$ with the
induced metric, then we obtain an invariant measure $\nu$ with 
$\supp(\nu) = \Ll$. 
\end{proof}

Note that unless we know that the SDS is locally contractive, we cannot argue
right away that every invariant measure must be supported exactly by $\Ll$.
The assumptions \eqref{eq:assume1} \& \eqref{eq:assume2} will in general not imply
local contractivity, as we shall see below. Thus, the question of uniqueness of
the invariant measure is more subtle. For a sufficient condition that requires 
a more restrictive %\& measure theoretic 
(Harris type) notion of irreducibility, see  
\cite[Def. 5.4 \& Thm. 5.5]{Li}.
%\\[8pt]
\section{Hyperbolic extension}\label{sec:hyp}
%\\[6pt]
In order to get closer to answering the uniqueness question in a more ``topological''
spirit, we also want to control the Lipschitz constants $A_n$. 
We shall need to distinguish between two cases.
\\[5pt]
\noindent
{\bf A. Non-lattice case}
\\[3pt]
If the random variables $\,\log A_n\,$ are non-lattice, i.e., there is no
$\kappa > 0$ such that $\,\log A_n \in \kappa\cdot\Z$ almost surely, then
we consider the extended SDS 
\begin{equation}\label{eq:Xntilde}
\wh X_n^{x,a}= (X_n^x,A_n A_{n-1} \cdots A_1 a) 
\end{equation}
on the extended space $\wh \Xx = \Xx \times \R^+_*$, with initial point 
$(x,a) \in \wh\Xx$. 
We also extend $\nu$ to a Radon measure $\lambda = \lambda_{\nu}$ on $\wh \Xx$ by
\begin{equation}\label{eq:lambdaA}
\int_{\wh \Xx} \varphi(x,a) \,d\lambda(x,a) = 
\int_{\Xx} \int_{\R} \varphi(x,e^u) \, d\nu(x)\,du\,.
\end{equation}
This is the product of $\nu$ with the multiplicative Haar measure on $\R^+_*$.
\\[5pt]
%\emph
{\bf B. Lattice case}
\\[3pt]
Otherwise, there is a maximal $\kappa > 0$ such that 
$\log A_n \in \kappa\cdot\Z$ almost surely. Then we consider again the extended
SDS \eqref{eq:Xntilde}, but now the extended space is 
$\wh \Xx = \Xx \times \exp(\kappa \cdot\Z)$, where of course 
$\exp(\kappa\cdot \Z) = \{ e^{\kappa\, m}: m \in \Z \}$. 
The initial point $(x,a)$ now has to be such that also $a \in \exp(\kappa \Z)$.
In this case, we define $\lambda$ by
\begin{equation}\label{eq:lambdaB}
\int_{\wh \Xx} \varphi(x,a) \,d\lambda(x,a) = 
\int_{\Xx} \sum_{m \in \Z} \varphi(x,e^{\kappa m}) \, d\nu(x)\,.
\end{equation}

In both cases, it is straightforward to verify that $\lambda$ is an invariant
Radon measure for the extended SDS on $\wh \Xx$.

%The extended process may be interesting also for further purposes.

\medskip

%[Here and in the remaining
%sections, $\wh{\ }$ does not refer to Fourier transformation.]

Consider the hyperbolic upper half plane $\HH \subset \C$ with
the Poincar\'e metric 
$$
\theta(z,w) = \log \frac{|z-\bar w| + |z-w|}{|z-\bar w| - |z-w|}\,,
$$
where $z, w \in \HH$ and $\bar w$ is the complex conjugate of $w$.
We use it to define a ``hyperbolic'' metric on $\wh X$ by
\begin{equation}\label{eq:hyp}
\begin{aligned}
\hat d\bigl((x,a),(y,b)\bigr) 
&= \theta\bigl(\im a, d(x,y) + \im b\bigr)\\  
&= \log \frac{\sqrt{d(x,y)^2 + (a+b)^2} +\sqrt{d(x,y)^2 + (a-b)^2}}
 {\sqrt{d(x,y)^2 + (a+b)^2} - \sqrt{d(x,y)^2 + (a-b)^2}}\,.
\end{aligned}
\end{equation} 
It is a good exercise, using the specific properties of $\theta$, to verify
that this is indeed a metric. The metric space $(\wh \Xx, \hat d)$
is again proper, and for any $a > 0$, the embedding $\Xx \to \wh \Xx\,$,
$x \mapsto (x,a)$, is a homeomorphism.

\begin{lem}\label{lem:lift-f}
Let $f : \Xx \to \Xx$ be a Lipschitz mapping with Lipschitz
constant $\lp(f) > 0$. Then the mapping $\hat f: \wh \Xx \to \wh \Xx$,
defined by 
$$
\hat f(x,a) = \bigl( f(x),\lp(f)a\bigr)
$$
is a contraction of $(\wh \Xx, \hat d)$ with Lipschitz constant $1$.
\end{lem} 

\begin{proof} We have by the dilation invariance of the hyperbolic metric
$$
\begin{aligned}
\tilde d\Bigl(\hat f(x,a),\hat f(y,b)\Bigr)
&= \theta\Bigl(\im \lp(f)a\,,\, d\bigl(f(x),f(y)\bigr) + \im \lp(f)b\Bigr)
\le \theta\Bigl(\im \lp(f)a\,,\, \lp(f)d(x,y) + \im \lp(f)b\Bigr)\\
&= \theta\Bigl(\im a\,,\, d(x,y) + \im b\Bigr)
= \wh d\Bigl((x,a),(y,b)\Bigr)\,.
\end{aligned}
$$
Thus, $\lp(\hat f) \le 1$. Furthermore, if $\ep > 0$ and $x,y \in \Xx$
are such that $d\bigl(f(x),f(y)\bigr) \ge (1-\ep)\lp(f)d(x,y)$ then 
we obtain in the same way that
$$
\hat d\Bigl(\tilde f(x,a),\tilde f(y,b)\Bigr)
\ge \theta\Bigl(\im a, (1-\ep) d(x,y) + \im b\Bigr)\,.
$$
when $\ep \to 0$, the right hand side tends to 
$\hat d\bigl((x,a),(y,b)\bigr)$. Hence $\lp(\hat f)=1$.
\end{proof}

Thus, with the sequence $(F_n)$, we associate the sequence $(\wh F_n)$
of i.i.d. Lipschitz contractions of $\wh \Xx$ with Lipschitz constants $1$.
The associated SDS on $\wh\Xx$ is $(\wh X_n^{x,a})$, as defined in
\eqref{eq:Xntilde}. From Lemma \ref{lem:transient}, which is true for any SDS of 
contractions, we get the following, where $o \in \Xx$ and $\hat o = (o,1)$.
\begin{cor}\label{cor:transient} 
$
\Prob\bigl[\hat d\bigl(\wh X_n^{x,a}, \hat o\bigr) \to \infty\bigr] 
\in \{0,1\}\,,
$
and the value is the same for all \hbox{$(x,a) \in \wh X$.}
\end{cor}
%$\,$\\[3pt]
\section{Transient extended SDS}\label{sec:trans-ext}
%\\[6pt]
We first consider the situation when $(\wh X_n^{x,a})$ is transient, i.e., 
the probability in Corollary \ref{cor:transient} is $=1$.
We shall use the comparison \eqref{eq:dominate}
of $(X_n^x)$ with the affine stochastic recursion $(Y_n^{|x|})$. Recall
that $|x| = d(o,x)$ and that $B_n \ge 0$. 
The hyperbolic extension $(\wh Y_n^{|x|,a})$ of $(Y_n^{|x|})$ 
is a random walk on the hyperbolic upper half plane. It can be also seen as a
random walk on the affine group of all mappings $g_{a,b}(z) = az+b$. 
Under the non-degeneracy assumptions of Proposition \ref{pro:centered},
this random walk is well-known to be transient.

\begin{lem}\label{lem:hitting}
Assume that \eqref{eq:assume1}, \eqref{eq:nofix} and \eqref{eq:assume2} hold.

Then for every sufficiently large $r > 0$ and every $s >1$ there are 
$\alpha= \alpha_{r,s}$ and $\delta = \delta_{r,s}> 0$ such that, setting 
$K_{r,s} = [0\,,\,r] \times [1/s\,,\,s]$
and $Q_{r,\alpha} = [0\,,\,r] \times [\alpha\,,\,\infty)$, 
one has for the affine recursion that
$$ 
\Prob[\wh Y_n^{y,a} \in K_{r,s} \; \text{for some}\; n \ge 1] \ge \delta
\quad \text{for all} \; (y,a) \in Q_{r,\alpha}\,.
$$
\end{lem}

\begin{proof} In this proof only, we write $\nu$ for the invariant Radon measure
associated with $(Y_n^{|x|})$. It existence is guaranteed by Proposition
\ref{pro:centered}. Let $\lambda = \lambda_{\nu}$ be its 
hyperbolic extension according to \eqref{eq:lambdaA}, resp. \eqref{eq:lambdaB}. 
We normalize $\nu$, and consequently $\lambda$, so that $\nu$ is the measure which 
is denoted $m(f)$ in \cite[p. 482]{BBE}.

The random walk $(\wh Y_n^{y,a})$ on the affine group (parametrized by 
$\R^+_* \times \R$) evolves on $\R^+_* \times \R^+$, when $y \ge 0$. 
By \cite{BBE}, its potential kernel
$$
\mathcal U\varphi(y,a) = \sum_{n=0}^{\infty} \Ex\bigl(\varphi(\wh Y_n^{y,a})\bigr)\,,
\quad \varphi \in \Ccal_c\bigl(\R^+_* \times \R^+\bigr)),
$$ 
is finite and weakly compact as a family of Radon measures that are parametrized
by $(y,a)$. Furthermore \cite[Thm. 2.2]{BBE}, 
$$
\lim_{a \to \infty}
\mathcal U\varphi(y,a) = \int \varphi\,d\lambda,
$$ 
and convergence is uniform when $y$ remains in a compact set.
We fix $r > 1$ large enough so that $\nu([0\,,\,r']) > 0$, where $r' = r-1$,
and let $s > 1$ be arbitrary. We set $s' = (s+1)/2$ and
$c_{r,s} = \lambda(K_{r',s'})/2$, which is strictly positive, and choose
$\varphi \in  \Ccal_c^+\bigl(\R^+_* \times \R^+\bigr)$ so that
$\uno_{K_{r',s'}} \le \varphi \le \uno_{K_{r,s}}\,$. By the above, there is 
$\alpha = \alpha_{r,s} > 0$ such that $\mathcal U\varphi(y,a) \ge c_{r,s}$ for all 
$(y,a) \in Q_{r,\alpha}\,$. Given any starting point $(y,a)$, let 
$$
\tau = \inf \{ n \ge 1: \wh Y_n^{y,a} \in K_{r,s} \}\,.
$$
We know that
$$
M_{r,s} = \sup \mathcal U\uno_{K_{r,s}} < \infty.
$$
Let $(y,a) \in Q_{r,\alpha}\,$. Just for the purpose of this proof, 
we consider the hitting distribution $\sigma_{(y,a)}$ on  $K_{r,s}$ defined by 
$\sigma_{(y,a)}(B) = \Prob[\tau < \infty\,,\; \wh Y_{\tau}^{y,a} \in B]$.
Then by the Markov property, 
$$
\begin{aligned}
\mathcal U\uno_{K_{r,s}}(y,a)
&= \Ex \Bigl(\sum_{n=0}^{\infty}\uno_{K_{r,s}}(\wh Y_n^{y,a})\Bigr)%\\
= \Ex \Bigl(\uno_{[\tau < \infty]} 
\sum_{n=\tau}^{\infty}\uno_{K_{r,s}}(\wh Y_n^{y,a})\Bigr)\\
&= \int_{K_{r,s}} \Ex \Bigl(
\sum_{n=0}^{+\infty}\uno_{K_{r,s}}(\wh Y_n^{z,b})\Bigr) \,d\sigma_{(y,a)}(z,b)\\
&\le M_{r,s}\,\sigma_{(y,a)}(K_{r,s}) =M_{r,s}\, \Prob_{(y,a)}[\tau < \infty],
\end{aligned} 
$$
where the index $(y,a)$ indicates the starting point.
Therefore we can set $\delta=M_{r,s}/c_{r,s}\,$, and 
$\Prob_{(y,a)}[\tau < \infty] \ge \delta$ for all
$(y,a) \in Q_{r,\alpha}\,$.
\end{proof}

Let $\overline \Bb(r)$ be the closed ball in $\Xx$ with center $0$ and radius $r$.
Set $\Bb_{r,s} = \overline \Bb(r)\times [1/s\,,\,s]$ and 
$\Cq_{r,\alpha} = \overline \Bb(r)\times [\alpha\,,\,\infty)$.

\begin{lem}\label{lem:rept}
Assume that \eqref{eq:assume1}, \eqref{eq:nofix}  and \eqref{eq:assume2} hold 
and that $(\wh X_n^{x,a})$ is transient.
Then for every sufficiently large $r > 0$, there is $\alpha > 0$ such that
$$
\Prob[\wh X_n^{x,a} \in \Cq_{r,\alpha} \; 
\text{for infinitely many}\,\; n] = 0 \quad\text{for all }\;(x,a) \in \wh X. 
$$
\end{lem}

\begin{proof}
Let 
$$
\Lambda = \Lambda^{x,a} = \{\omega \in \Omega : \wh X_n^{x,a}(\omega) 
\in \Cq_{r,\alpha} \; \text{for infinitely many}\; n\}.
$$ 
Given $r$ sufficiently large so that Lemma \ref{lem:hitting} applies,
choose $s > 1$ and let $\alpha$ and $\delta > 0$ be as in that lemma.
For each $(c,a) \in Q_{r,\alpha}$ there is an 
index $N_{c,a} \in \N$ such that
\begin{equation}\label{eq:hitting}
\Prob[\wh Y_n^{y,a} \in K_{r,s} \; \text{for some $n$ with}\; 1 \le n \le N_{c,a}] 
\ge \delta/2.
\end{equation}
If $(c,a) \notin Q_{r,\alpha}$ then we set $N^{c,a}=0$.
Since $\Bb_{r,s}$ is compact, the transience assumption yields that
$\Prob\bigl(\bigcup_{j=2}^{\infty} \Omega_j\bigr) = 1$, where
$$
\Omega_j= \Omega_j^{x,a}
= \{\omega \in \Omega : \wh X_n^{x,a}(\omega) \notin \Bb_{r,s} \;\text{for every} 
\; n \ge j \}.
$$
Thus, we need to show that $\Prob(\Lambda \cap \Omega_j) = 0$ for every $j \ge 2$.
We define a sequence of stopping times $\tau_k=\tau_k^{x,a}$ and (when $\tau_k <
\infty$) associated pairs $(x_k,a_k) =\wh X_{\tau_k}^{x,a}$ by
$$
\begin{aligned}
\tau_1 &= \inf \{ n > N^{|x|,a}: \wh X_n^{x,a} \in \Cq_{r,\alpha}\} \AND\\
\tau_{k+1} 
&= \begin{cases}
 \inf \{ n > \tau_k + N^{|x_k|,a_k}: \wh X_n^{x,a} \in \Cq_{r,\alpha}\}\,,
 &\text{if} \; \tau_k < \infty\,,\\ 
 \infty\,,&\text{if} \; \tau_k = \infty\,.
 \end{cases}
\end{aligned} 
$$
Unless explained separately, we always use $\tau_k = \tau_k^{x,a}$.
Note that $\omega \in \Lambda$ if and only if $\tau_k(\omega) < \infty$
for all $k$. Therefore
$$
\Lambda \cap \Omega_j = \bigcap_{k \ge j} \Lambda_{j,k}\,,\quad
\text{where}\quad \Lambda_{j,k} = [ \tau_k < \infty\,,\;  
\wh X_n^{x,a} \notin \Bb_{r,s}\; \text{for all $n$ with}\; j \le n \le \tau_k ].
$$
We have $\Lambda_{j,k} \subset \Lambda_{j,k-1}\,$.
Next, note that 
$$
\text{if} \quad \wh X_n^{x,a}(\omega) \notin \Bb_{r,s} \quad
\text{then} \quad \wh Y_n^{|x|,a}(\omega) \notin K_{r,s}\,. 
$$
This follows from \eqref{eq:dominate}. 

We have that 
$\wh X_{\tau_{k-1}}^{x,a} \in \Cq_{r,\alpha}$ for $k \ge 2$.
Just for the purpose of the next lines of the proof, we introduce the measure 
$\sigma$ on $\Cq_{r,\alpha}$ given by $\sigma(\wh B) 
= \Prob\bigl(\Lambda_{j,k-1} \cap [\wh X_{\tau_{k-1}}^{x,a} \in \wh B]\bigr)$,
where $\wh B \subset \Cq_{r,\alpha}$ is a Borel set.
Then, using the strong Markov property and \eqref{eq:hitting}, 
$$
\begin{aligned}
\Prob(\Lambda_{j,k}) &= \Prob\Bigl( [\tau_k < \infty\,,\;  
\wh X_n^{x,a} \notin \Bb_{r,s}\; \text{for all $n$ with}\; 
\tau_{k-1} < n \le \tau_k] \cap \Lambda_{j,k-1}] \Bigr)\\
&= \int_{\Cq_{r,\alpha}}
\Prob[\tau_1^{y,b} < \infty\,,\; \wh X_n^{y,b} \notin \Bb_{r,s}\; 
\text{for all $n$ with}\;0 < n \le \tau_1^{y,b}\,]\,d\sigma(y,b)\\
&\le \int_{\Cq_{r,\alpha}}
\Prob[\tau_1^{y,b} < \infty\,,\; \wh Y_n^{|y|,b} \notin K_{r,s}\; 
\text{for all $n$ with}\;0 < n \le N^{|y|,b}\,]\,\,d\sigma(y,b)\\
&\le \int_{\Cq_{r,\alpha}} (1-\delta/2)\,\,d\sigma(y,b) =
(1-\delta/2)\,\Prob(\Lambda_{j,k-1})\,.  
\end{aligned}
$$
We continue recursively downwards until we reach $k=2$ (since $k=1$ is excluded
unless $(x,a) \in \Cq_{r,\alpha}$). Thus, $\Prob(\Lambda_{j,k}) 
\le (1-\delta/2)^{k-1}$,
and as $k \to \infty$, we get $\Prob(\Lambda \cap \Omega_j) = 0$, as required.
\end{proof}

\begin{thm}\label{thm:lip-trans} 
Given the random i.i.d. Lipschitz mappings $F_n\,$, let $A_n$ and $B_n$ be as in 
\eqref{eq:lip-an-bn}. Suppose that \eqref{eq:assume1}, \eqref{eq:nofix}  
and \eqref{eq:assume2} hold, and that
$\Prob\bigl[\hat d\bigl(\wh X_n^{x,a}, \hat o\bigr) \to \infty\bigr]=1$.
Then the SDS induced by the $F_n$ on $\Xx$ is locally contractive.

In particular, it has an invariant Radon 
measure $\nu$ that is unique up to multiplication with constants. 

Also, the shift $T$ on  
$\bigl(\Xx^{\N_0}, \Bf(\Xx^{\N_0}), \Prob_{\nu}\bigr)$ 
is ergodic, where $\Prob_{\nu}$ is the measure on $\wh \Xx^{\N_0}$. 
associated with $\nu$.
\end{thm}

\begin{proof} 
Fix any starting point $(x,a)$ of the extended SDS. Let $r$ be sufficiently
large so that the last two lemmas apply, and such that
$$
\Prob[ X_n^x \in \overline \Bb(r) \; \text{for infinitely many}\,\,n] = 1.
$$
We claim that
\begin{equation}\label{eq:liploc} 
\lim_{n \to \infty} A_{0,n}\,\uno_{\overline \Bb(r)}(X_n^x) = 0 \quad
\text{almost surely.}
\end{equation}
We consider $\alpha$ associated with $r$ as in  Lemma \ref{lem:rept}. Then we choose
an arbitrary $s \ge \alpha$.
We know by transience of the extended SDS that 
$$
\Prob[ \wh X_n^{x,a} \in \Bb_{r,s} \; \text{for infinitely many}\,\,n] = 0.
$$
We combine this with Lemma \ref{lem:rept} and get
$$
\Prob[ \wh X_n^{x,a} \in \Bb_{r,s} \cup \Cq_{r,\alpha} \; 
\text{for infinitely many}\,\,n] = 0.
$$
Since $s \ge \alpha$, we have 
$\Bb_{r,s} \cup \Cq_{r,\alpha} = \overline \Bb(r) \times [1/s\,,\,\infty)$.

Thus,  if $\N(x,r)$ denotes the a.s. 
infinite random set of all $n$ for
which $X_n^x \in \overline \Bb(r)$, then for all but finitely many
$n \in \N(x,r)$, we have $A_{0,n} < 1/s$. This holds for every $s > \alpha$, and
we have proved \eqref{eq:liploc}.
We conclude that
$$
d(X_n^x\,,X_n^y)\, \uno_{\overline \Bb(r)}(X_n^x)  
\le A_{0,n} \,d(x,y)\,\uno_{\overline \Bb(r)}(X_n^x) \to 0 \quad\text{almost
surely.}
$$
Now that we have local contractivity, the remaining statements follow
from Theorem \ref{thm:uniquemeasure}.
\end{proof}
%$\,$\\[3pt]
\section{Conservative extended SDS}\label{sec:cons-ext}
%\\[6pt]
Now we assume to be in the conservative case, i.e., the probability in 
Corollary \ref{cor:transient} is $=0$. We  start with an invariant measure 
$\nu$ for the SDS  on $\Xx$. 
If \eqref{eq:assume1},\eqref{eq:nofix}   \& \eqref{eq:assume2} hold,
its existence is guaranteed by Lemma \ref{lem:inv-exist}. Then we extend
$\nu$ to the measure $\lambda =\lambda_{\nu}$ on $\wh \Xx$ of
\eqref{eq:lambdaA}, resp. \eqref{eq:lambdaB}.

We can realize the extended SDS, starting at $(x,a) \in \wh \Xx$, on the space 
$$
\bigl(\wh \Xx^{\N_0}, \Bf(\wh \Xx^{\N_0}), \Prob_{x,a}\bigr),
$$
where $\Bf(\wh \Xx^{\N_0})$ is the product Borel $\sigma$-algebra,
and $\Prob_{x,a}$ is the image of the measure $\Prob$ under the mapping
$$
\Omega \to \wh \Xx^{\N_0}\,,\quad 
\omega \mapsto \bigl( \wh X_n^{x,a}(\omega)\bigr)_{n \ge 0}\,.
$$
Then we consider the Radon measure on $\wh \Xx^{\N_0}$ defined by
$$
\Prob_{\lambda} = \int_{\wh \Xx} \Prob_{x,a}\,d\lambda(x,a).
$$
The integral with respect to $\Prob_{\lambda}$ is denoted $\Ex_{\lambda}\,$.
We write $\wh T$ for the time shift on $\wh \Xx^{\N_0}$. Since $\lambda$
is invariant for the extended SDS, $\wh T$ is a contraction of 
$L^1(\wh \Xx^{\N_0}, \Prob_{\lambda})$. Also, in this section, 
$\If$ stands for the $\sigma$-algebra of the
$\wh T$-invariant sets in $\Bf(\wh \Xx^{\N_0})$.
As before, any function $\varphi: \wh \Xx^{\ell} \to \R$ is extended to 
$\wh \Xx^{\N_0}$ by setting $\varphi(\mathbf{x},\mathbf{a}) 
= \varphi\bigl((x_0,a_0), \dots, (x_{\ell-1}, a_{\ell-1})\bigr)$,
if $(\mathbf{x},\mathbf{a}) = \bigl((x_n,a_n)\bigr)_{n \ge 0}\,$.
In analogy with \eqref{eq:Xmn}, we define 
$$
\wh X_{m,n}^{x,a} = \bigl(X_{m,n}^x\,,\, A_{m,n}a\bigr)\quad (n \ge m)\,. 
$$
We now set for $n \ge m$ and $\varphi: \wh \Xx^{\N_0} \to \R$
$$
S_{m,n}^{x,a}\varphi(\omega) = \sum_{k=m}^{n} 
\varphi\Bigl(\bigl(\wh X_{m,k}^{x,a}(\omega)\bigr)_{k \ge m}\Bigr) 
$$
and in particular $S_n^{x,a}\varphi(\omega) = S_{0,n}^{x,a}\varphi(\omega)$. 
Consider the sets 
\begin{equation}\label{eq:Omegar}
\Omega_r= \bigl\{ \omega \in \Omega: 
\liminf \hat d\bigl(\wh X_n^{\hat o}(\omega), \hat o\bigr) \le r \bigr\}\;\; 
(r \in\N) \AND \Omega_{\infty} = \bigcup_r  \Omega_r\,.
\end{equation}
By our assumption of conservativity, $\Prob(\Omega_{\infty})=1$.
For $r \in \N$, write 
$\wh \Bb(r)$ for the \emph{closed} ball in $(\wh\Xx,\hat d)$ with center $\hat o$
and radius $r$.
Then for every $\omega \in \Omega_r$ and $s \in \N_0\,$, the set
$\{ n : \wh X_n^{x,a}(\omega) \in \wh \Bb(r+s)\;\text{for all}\; 
(x,a) \in \wh \Bb(s)\}
$
is infinite. For each $r$, set 
$\psi_r(x,a) = \max \bigl\{ 1 - \hat d\bigl((x,a),\wh \Bb(r)\bigr)\,,\,0\}$.
Then $\psi_r \in \Ccal_c^+(\wh\Xx)$ satisfies 
\begin{equation}\label{eq:psi-prop}
\begin{gathered} 
\uno_{\wh \Bb(r+1)} \ge \psi_r \ge \uno_{\wh \Bb(r)}\,,\\
|\psi(x,a) - \psi(y,b)| \le \hat d\bigl((x,a),(y,b)\bigr)\; \text{on}\; \wh\Xx\,,
\AND\\
S_n^{x,a} \psi_{r+s}(\omega) \to \infty \quad \text{for all}\; 
\omega \in \Omega_r\,,\; (x,a) \in \wh \Bb(s)\,.
\end{gathered}
\end{equation}
Then we can find a decreasing sequence of numbers $c_r > 0$ such that
$\sum_r c_r \max \psi_{r+2} < \infty$ and the functions 
\begin{equation}\label{eq:PhiPsi}
\Phi = \sum_r c_r\,\psi_{r+2}\AND \Psi = \sum_r c_r\,\psi_r
\end{equation} 
are in  $L^1(\wh\Xx,\lambda)$ and thus (there extensions to $\Xx^{\N_0}$) in
$L^1(\wh\Xx^{\N_0},\Prob_{\lambda})$. They will be used below several times.
Both are continuous and strictly positive on $\wh\Xx$, and by construction,
$$
\sum_n \Psi\bigl(\wh X_n^{x,a}(\omega)\bigr) = \infty
\quad \text{for all}\; \omega \in \Omega_{\infty}\;
\text{and}\; (x,a) \in \wh \Xx\,.
$$
We have obtained the following.

\begin{lem}\label{lem:cons} When the extended SDS is conservative, 
$\wh T$ is conservative.
\end{lem}

Next, for any $\varphi\in L^1(\wh \Xx^{\N_0},\Prob_{\lambda})$, consider the function
$\vb_{\varphi} = \Ex_{\lambda}(\varphi\,|\,\If)/\Ex_{\lambda}(\Psi\,|\,\If)$ on
$\wh\Xx^{\N_0}$. A priori, the quotient of conditional expectations 
is defined only $\Prob_{\lambda}$-almost everywhere,
and we consider a representative which is always finite. 
We turn this into the family of finite positive random variables
$$
V_{\varphi}^{x,a}(\omega) = 
\vb_{\varphi}\Bigl(\bigl(\wh X^{x,a}_n(\omega)\bigr)_{n \ge 0}\Bigr),\quad
(x,a) \in \wh\Xx.
$$
\begin{lem}\label{lem:ratio} In the conservative case, let 
$\tau : \Omega \to \N$ be any a.s. finite
random time. Then, on the set where $\tau(\omega) < \infty$,
for every $\varphi\in L^1(\wh \Xx^{\N_0},\Prob_{\lambda})$, 
$$
\lim_{n \to \infty} \frac{S_n^{x,a}\varphi-S_{\tau}^{x,a}\varphi}
{S_n^{x,a}\Psi-S_{\tau}^{x,a}\Psi} 
%= \frac{\Ex_{\lambda}(\varphi\mid\If)}{\Ex_{\lambda}(\Psi\mid\If)} \quad 
= V_{\varphi}^{x,a} \quad
\Prob\text{-almost surely}\,,\;\text{for}\; 
\lambda\text{-almost every}\; (x,a) \in \wh \Xx.
$$
\end{lem}

%%%\comment{Comment for internal use: the terminology
%%%$$
%%%\frac{S_n^{x,a}\varphi}{S_n^{x,a}\psi} \to
%%%\frac{\Ex_{\lambda}(\varphi\mid\If)}{\Ex_{\lambda}(\psi\mid\If)}
%%%\quad\text{almost surely}
%%%$$
%%%had confused me for a while. Indeed, here the left hand side is
%%%a quotient of 2 random variables = functions of $\omega \in \Omega$,
%%%the basic probability space. The right hand side is in principle an 
%%%$\If$-measurable function on $\wh\Xx^{\N_0}$, so it appears to live on a
%%%different space. What is meant is of course that this function on the
%%%right hand side is evaluated in the element of $\wh\Xx^{\N_0}$ which
%%%is the image of $\omega$, when the starting point is $(x,a)$. This is
%%%just the trajectory $\bigl(\wh X_n^{x,a} (\omega)\bigr)_{n\ge 0}\,.$}

\begin{proof} We know that 
$S_n^{x,a}\Psi(\omega)  \to \infty$ for all $\omega \in \Omega_{\infty}$. 
Once more by the Chacon-Ornstein theorem, 
$S_n^{x,a}\varphi/S_n^{x,a}\Psi \to V_{\varphi}^{x,a}$
almost surely on $\Omega_{\infty}\,$, for $\lambda$-almost every 
$(x,a) \in \wh \Xx$. 
Furthermore,  both 
$S_{\tau}^{x,a}\varphi/S_n^{x,a}\Psi$ and 
$S_{\tau}^{x,a}\Psi/S_n^{x,a}\Psi$ tend
to $0$ on $\Omega_{\infty}\,$, as $n \to \infty\,$. When $n > \tau$,
$$
\frac{S_n^{x,a}\varphi}{S_n^{x,a}\Psi} = 
\underbrace{\frac{S_{\tau}^{x,a}\varphi}{S_n^{x,a}\Psi}}_{\displaystyle
 \to 0 \;\text{a.s.}}
+ \biggl( 1 - \underbrace{\frac{S_{\tau}^{x,a}\Psi}{S_n^{x,a}\Psi}}_{\displaystyle 
 \to 0\;\text{a.s.}}\biggr)
\frac{S_n^{x,a}\varphi-S_{\tau}^{x,a}\varphi}
{S_n^{x,a}\Psi-S_{\tau}^{x,a}\Psi}\,. 
$$
The statement follows.
\end{proof}

When the extended SDS is conservative, 
%i.e., $\Prob\bigl[\hat d\bigl(\wh X_n^{x,a}, \hat o\bigr) \to \infty\bigr]=0$, 
we do not see how to
involve local contractivity, but we can provide a reasonable additional 
assumption which will yield %topological recurrence of $(X_n)$ and 
uniqueness of the invariant Radon measure. We set
\begin{equation}\label{eq:Zn} 
D_n(x,y) = \frac{d(X_n^x,X_n^y)}{A_1\cdots A_n}\,.
\end{equation}
(Compare with the proof of Theorem \ref{thm:contractive}, which corresponds to 
$A_n \equiv 1$.)
The assumption is 
\begin{equation}\label{eq:assume3}
\Prob[D_n(x,y) \to 0] =1 \quad \text{for all}\; x, y \in \Xx.
\end{equation}

\begin{rmk}\label{rmk:inside}
If we set $D_{m,n}(x,y) = d(X_{m,n}^x,X_{m,n}^y)/A_{m,n}$ 
then \eqref{eq:assume3} implies that
$$
\Prob\left[\lim_{n \to \infty} D_{m,n}(x,y) = 0 \; \text{for all}\; 
x, y \in \Xx\,,\; m \in \N\right] 
= 1.
$$
Indeed, let $\Xx_0$ be a countable, dense subset of $\Xx$. Then \eqref{eq:assume3} 
implies that
$$
\Prob\left[\lim_{n \to \infty} D_{m,n}(x,y) = 0 \; \text{for all}\; 
x, y \in \Xx_0\,,\; m \in \N\right] 
= 1.
$$
Let $\Omega_0$ be the subset of $\Omega_{\infty}$ where this holds. 

Note that $D_{m,n}(x,y) \le d(x,y)$.
Given arbitrary $x, y \in \Xx$ and $x_0, y_0 \in \Xx_0\,$, we get on~$\Omega_0$
$$
D_{m,n}(x,y) \le D_{m,n}(x_0,y_0) + d(x,x_0) + d(y,y_0)\,,
$$
and the statement follows.\qed
\end{rmk}

In the next lemma, we give a condition for \eqref{eq:assume3}. It will be 
useful, in \S \ref{sec:ref-aff}.

\begin{lem}\label{lem:Dnto0}
In the case when the extended SDS is conservative,
suppose that for every $\ep > 0$ and $r \in \N$ there is $k$ such that 
$\Prob[D_k(x,y) < \ep \;\text{for all}\; x,y \in \Bb(r)] > 0$.
Then \eqref{eq:assume3} holds.
\end{lem}

\begin{proof}
We set $D_{\infty}(x,y)= \lim_n D_n(x,y)$ and 
$w(x,y) = \Ex\bigl(D_{\infty}(x,y)\bigr)$. 
A straightforward adaptation of the argument used in the proof  of Theorem 
\ref{thm:contractive} yields that
\begin{equation}\label{eq:martingale2}
\lim_{m \to \infty} \frac{w(X_m^x\,,X_m^y)}{A_1 \cdots A_m} = D_{\infty}(x,y)
\quad\text{almost surely.}
\end{equation}
Again, we claim that $\Pr[D_{\infty}(x,y) \ge \ep] = 0$. 
By conservativity, it is sufficient to show that
$\Pr(\Lambda_r)=0$ for every $r \in  \N$, where
$$
\Lambda_r = \bigcap_{m \ge k} \bigcup_{n \ge m} 
[\wh X_n^x\,,\; \wh X_n^y \in \Bb(r) \times [1/r\,,\,r]\,,\;D_n(x,y) \ge \ep]\,.
$$
By assumption, there is $k$ such that the event 
$\Gamma_{k,r} = [D_k(x,y) < \ep/2 \;\text{for all}\; x,y \in \Bb(r)]$ 
satisfies $\Pr(\Gamma_{k,r})>0$. 

We now continue as in the proof of Theorem \ref{thm:contractive}, and find that
for all $u, v \in \Bb(r)$ with $d(u,v) \ge \ep$,
$$
w(u,v) \le d(u,v) - \delta\,,\quad\text{where} \quad \delta 
= \Pr(\Gamma_{k,r})\cdot (\ep/2) > 0.
$$
This yields that on $\Lambda_r\,$, almost surely we have infinitely many $n \ge k$ 
for which 
$w(X_n^x,X_n^y) \le d(X_n^x,X_n^y) - \delta$ and $A_1 \cdots A_n \le r$,
that is,
$$
\frac{w(X_n^x,X_n^y)}{A_1 \cdots A_n} \le D_n(x,y) - \frac{\delta}{r}
\quad \text{infinitely often.}
$$
Letting $n \to \infty$, we get $D_{\infty}(x,y) < D_{\infty}(x,y)$ almost
surely on $\Lambda_r\,$, so that indeed $\Prob(\Lambda_r)=0$. 
\end{proof}

We now elaborate the main technical prerequisite for handling the case when
the extended SDS in conservative. 
Some care may be in place to have a clear 
picture regarding the dependencies of sets on which various ``almost everywhere'' 
statements hold. Let $\varphi \in L^1(\wh \Xx^{\N_0},\Prob_{\lambda})$.
Let $\Omega_0$ be as in Remark \ref{rmk:inside}.
For $\lambda$-almost every $(x,a) \in \wh\Xx$, there is a set 
$\Omega^{x,a}_{\varphi} \subset \Omega_0$ with 
$\Prob(\Omega^{x,a}_{\varphi})=1$, such that 
$$
\frac{S_n^{x,a}\varphi(\omega)}{S_n^{x,a}\Psi(\omega)} \to V_{\varphi}^{x,a}(\omega)
$$
for every $\omega \in \Omega^{x,a}_{\varphi}$. For the remaining $(x,a) \in \wh\Xx$,
we set $\Omega^{x,a}_{\varphi}=\emptyset$.

\begin{pro}\label{pro:delta} In the case when the extended SDS is conservative, 
assume \eqref{eq:assume3}.
Let $\varphi \in\Ccal_c^+(\wh\Xx^{\ell})$ with $\ell \ge 1$. 
Then for every $\ep >0$ there is $\delta = \delta(\ep,\varphi) > 0$ with the 
following property. 

For all $(x,a), (y,b) \in \wh \Xx$ and any 
a.s. finite random time $\tau : \Omega \to \N_0$, one has on the set of all 
$\omega \in \Omega^{x,a}_{\Phi}$  with $\tau (\omega) < \infty$ and 
$\bigl|\log\bigl(A_{0,\tau}(\omega)a/b\bigr)\bigr| < \delta$ that
$$
\limsup_{n \to \infty} \left| \frac{S_n^{x,a}\varphi}{S_n^{x,a}\Psi}
                           - \frac{S_{\tau,n}^{y,b}\varphi}{S_{\tau,n}^{y,b}\Psi}
			    \right|
\le \ep \, W^{x,a}\,,
$$
where $W^{x,a}=V_{\Phi}^{x,a} + 1$.
\end{pro}

\begin{proof} 
Recall that $\Phi$, $\Psi$, $\varphi$ and $\psi_r$ are also considered
as functions on $\Xx^{\N_0}$ via their extensions defined above.

Since $\Psi$ is continuous and $> 0$, there is $C = C_{\varphi} > 0$
such that $\varphi \le C\cdot \Psi$. 
Also, there is some $r_0 \in \N$ such that the projection of $\supp(\varphi)$ 
onto the first coordinate in $\wh\Xx$ (i.e., the one with index $0$) is contained 
in $\wh\Bb(r_0)$.
We let $\ep' = \min\{ \ep/2, \ep/(2C), c_{r_0+1}\ep/2, 1 \}$, where $c_{r_0+1}$ 
comes from the definition \eqref{eq:PhiPsi} of $\Phi$ and $\Psi$.
Since $\varphi$ is uniformly continuous, there is $\delta > 0$ with
$2\delta \le \ep'$ such that 
$$
\begin{gathered}
\bigl|\varphi\bigl((x_0,a_0), \dots,(x_{\ell-1},a_{\ell-1})\bigr) 
- \varphi\bigl((y_0,b_0), \dots,(y_{\ell-1},b_{\ell-1})\bigr)\bigr| \le \ep'
\\
\text{whenever}\quad
\hat d\bigl((x_j,a_j),(y_j,b_j)\bigr) < 2\delta\,,\; j=0, \dots,\ell-1. 
\end{gathered}
$$
We write
$$ 
\left| \frac{S_n^{x,a}\varphi}{S_n^{x,a}\Psi}
- \frac{S_{\tau,n}^{y,b}\varphi}{S_{\tau,n}^{y,b}\Psi} \right|
\le \underbrace{\frac{|S_n^{x,a}\varphi-S_{\tau,n}^{y,b}\varphi|}{S_n^{x,a}\Psi}}_{
 \displaystyle \text{Term 1}} + 
\underbrace{\frac{S_{\tau,n}^{y,b}\varphi}{S_{\tau,n}^{y,b}\Psi}}_{
\displaystyle \le C_{\varphi}} 
\,\underbrace{\frac{|S_n^{x,a}\Psi-S_{\tau,n}^{y,b}\Psi|}{S_n^{x,a}\Psi}}_{
 \displaystyle \text{Term 2}}\,.
$$
We consider the random element $z = X_{\tau}^x\,$, so that $X_n^x=X_{\tau,n}^z$. 
Using the dilation invariance of hyperbolic metric,
$$
\begin{aligned}
\hat d(\wh X_n^{x,a}\,,\wh  X_{\tau,n}^{y,b}) 
&= \theta\bigl(\im A_{0,n}a\,,\, d(X_{\tau,n}^z\,,X_{\tau,n}^y) +  
\im A_{\tau,n}b\bigr)\\
&= \theta\bigl(\im A_{0,\tau}a\,, D_{\tau,n}(z,y) +  \im b\bigr)
\le |\log(A_{0,\tau}a/b)| +  D_{\tau,n}(z,y) +  \im b\bigr)
\,.
\end{aligned}
$$
By \eqref{eq:assume3}, for $\omega \in \Omega^{x,a}_{\Phi}$  with 
$\tau (\omega) < \infty$ there is a finite $\sigma(\omega) \ge \tau(\omega)$
in $\N$  such that $\theta\bigl(\im a, D_{\tau,n}(z,y) +  \im a\bigr) < \delta$ 
for all $n \ge \sigma(\omega)$. 
In the sequel, we assume that our $\omega \in \Omega^{x,a}_{\Phi}$ also satisfies 
$\bigl|\log\bigl(A_{0,\tau}(\omega)a/b\bigr)\bigr| < \delta$.

Now, we first bound the $\limsup$ of Term 1 by $\ep/2$.
If $n \ge \sigma$ and $|A_{0,\tau}(\omega)a/b| < \delta$, then we obtain
that 
$$
\bigl|
\varphi\bigl(\wh X_n^{x,a}\,, \wh X_{n+1}^{x,a}\,,\dots, 
\wh X_{n+\ell-1}^{x,a}\bigr) -
\varphi\bigl(\wh X_{\tau,n}^{y,b}\,, \wh X_{\tau,n+1}^{y,b}\,,\dots, 
\wh X_{\tau,n+\ell-1}^{y,b}\bigr) \bigr| < \ep' \le \ep/2.
$$
Suppose in addition that at least one of the two values
$\varphi\bigl(\wh X_n^{x,a}, \wh X_{n+1}^{x,a}\,,\dots, \wh X_{n+\ell-1}^{x,a}\bigr)$
or 
$\varphi\bigl(\wh X_{\tau,n}^{y,b}\,, \wh X_{\tau,n+1}^{y,b}\,,\dots, 
\wh X_{\tau,n+\ell-1}^{y,b}\bigr)$
is positive. Then at least one of 
$\wh X_n^{x,a}$ or $\wh X_{\tau,n}^{y,b}$ belongs to $\wh \Bb(r_0)$, and by the 
above  (since $\delta < 1$) both belong to $\wh \Bb(r_0+1)$. Thus, for 
$n \ge \sigma$,
$$
\begin{aligned}
\bigl|
\varphi\bigl(\wh X_n^{x,a}\,, \wh X_{n+1}^{x,a}\,,\dots, 
\wh X_{n+\ell-1}^{x,a}\bigr) -
\varphi\bigl(\wh X_{\tau,n}^{y,b}\,, \wh X_{\tau,n+1}^{y,b}\,,\dots, 
\wh X_{\tau,n+\ell-1}^{y,b}\bigr) \bigr|
&\le \ep'\,\psi_{r_0+1}\bigl(\wh X_n^{x,a}\bigr)\\
&\le (\ep/2) \,\Psi\bigl(\wh X_n^{x,a}\bigr).
\end{aligned}
$$
We get 
$$
\frac{\bigl|\bigl(S_n^{x,a}\varphi- S_{\sigma}^{x,a}\varphi\bigr)-
\bigl(S_{\tau,n}^{y,b}\varphi - S_{\tau,\sigma}^{y,b}\varphi\bigr)\bigr|}
{S_n^{x,a}\Psi- S_{\sigma}^{x,a}\Psi}
 \le \ep/2.
$$
Since $S_n^{x,a}\Psi \to \infty$ almost surely, when passing to the $\limsup$,
we can omit all terms in the last inequality that contain a $\sigma$; see
Lemma \ref{lem:ratio}.
This yields the bound on the $\limsup$ of Term 1.

Next, we bound the $\limsup$ of Term 2 by $\ep/2$. We start in the same
way as above, replacing $\varphi$ with an arbitrary one among the functions
$\psi_r$ and replacing $\ell$ with $1$. Using the specific properties 
\eqref{eq:psi-prop} of $\psi_r$
(in particular, Lipschitz continuity with constant $1$), and replacing
$\wh \Bb(r_0)$ with $\wh \Bb(r+1) = \supp(\psi_r)$, we arrive at the inequality
$$
\bigl|\psi_r\bigl(\wh X_n^{x,a}\bigr) -\psi_r\bigl(\wh X_{\tau,n}^{y,b}\bigr) \bigr|
\le \frac{\ep}{2C}\,\psi_{r+2}\bigl(\wh X_n^{x,a}\bigr).
$$
It holds for all $n \ge \sigma$, with probability $1$. We deduce  
$$
\bigl|\Psi\bigl(\wh X_n^{x,a}\bigr) -\Psi\bigl(\wh X_{\tau,n}^{y,b}\bigr) \bigr|
\le \frac{\ep}{2C}\,\Phi\bigl(\wh X_n^{x,a}\bigr)
$$
and 
$$
\frac{\bigl|\bigl(S_n^{x,a}\Psi- S_{\sigma}^{x,a}\Psi\bigr)-
\bigl(S_{\tau,n}^{y,b}\Psi - S_{\tau,\sigma}^{y,b}\Psi\bigr)\bigr|}
{S_n^{x,a}\Psi- S_{\sigma}^{x,a}\Psi}
\le \frac{\ep}{2C}\,\frac{S_n^{x,a}\Phi- S_{\sigma}^{x,a}\Phi}
{S_n^{x,a}\Psi- S_{\sigma}^{x,a}\Psi}
$$
Passing to the lim sup as above, and using the Chacon-Ornstein theorem here,
we get that the $\limsup$ of Term 2 is bounded almost
surely by $\frac{\ep}{2C}\, V_{\Phi}^{x,a}$.
\end{proof}
%\begin{pro}\label{pro:epsilon} In the conservative case, assume \eqref{eq:assume3}.
%Let $r \in \N$ and $\varphi \in \Ccal_c^+(\wh\Xx)$ supported by 
%$\wh B\bigl(m(r)+1\bigr)$. 

In the sequel, when we sloppily say ``for almost every $a > 0$'', we shall mean
``for Lebesgue-almost every $a > 0$'' in the non-lattice case,
resp. ``for every $a = e^{-\kappa m}$ ($m \in \Z$)'' in the lattice case.

\begin{cor}\label{cor:ally} Let $\varphi \in\Ccal_c^+(\wh\Xx^{\ell})$ as
above. For almost every $a > 0$, there is
a set $\Omega^a_{\varphi} \subset \Omega_0$ with
$\Prob(\Omega^a_{\varphi}) = 1$ such that for all $x,y \in \Xx$, 
$$
V^{x,a}_{\varphi} = V^{y,a}_{\varphi} =:V^a_{\varphi}.
$$
\end{cor}

\begin{proof} 
For almost every $a$, there is at least one 
$x_a \in \Xx$ such that $\Prob(\Omega^{x_a,a}_{\varphi})=1$. We can apply
Proposition \ref{pro:delta} with arbitrary $y \in \Xx$, 
$b=a$ and $\tau = 0$. Then we are allowed to take any $\ep > 0$ and get that 
$V^{x,a}_{\varphi} = V^{y,a}_{\varphi}$ on 
$\Omega^{x_a,a}_{\varphi} \cap  \Omega^{x_a,a}_{\Phi}\,$.
\end{proof} 

\begin{pro}\label{pro:constant}
Suppose that \eqref{eq:assume1}, \eqref{eq:nofix},  \eqref{eq:assume2} and 
\eqref{eq:assume3} hold, and that the extended SDS is conservative.
Let $\varphi \in\Ccal_c^+(\wh\Xx^{\ell})$, as above. Then for almost every 
$a > 0$, the random variable $V_{\varphi}^a$ is almost surely constant
(depending on $\varphi$ and -- so far -- on $a$). 
\end{pro}

\begin{proof} 
Let $a$ be such that $\Pr(\Omega^a_{\varphi})=1$, and choose $x=x_a$ as
in the proof of Corollary \ref{cor:ally}.

For $s \in \N$, let $\ep_s = 1/s$ and $\delta_s = \delta(\ep_s\,,\varphi)$ according
to  Proposition \ref{pro:delta}. By our assumptions, $(A_{0,n})_{n \ge 1}$
is a topologically recurrent random walk on $\R^+_*$, starting at $1$.
Choose $m \in \N$ and let $\tau_{m,s}$ be the $m$-th return time to the
interval $(e^{-\delta_s}\,,\,e^{\delta_s})$. For every $m$ and $s$, 
this is an almost surely finite stopping time, and we can find 
$\bar\Omega^a_{\varphi} \subset \Omega^a_{\varphi} \cap \Omega^{x,a}_{\Phi}$ with 
$\Prob(\bar\Omega^a_{\varphi}) = 1$ such that all $\tau_{m,s}$ are finite on 
that set.

We now apply Proposition \ref{pro:delta} with $(y,b)=(x,a)$ and $\tau = \tau_{m,s}$.
Then 
$$
\limsup_{n \to \infty} \biggl| V^a\varphi 
- \underbrace{\frac{S_{\tau,n}^{x,a}\varphi}{S_{\tau,n}^{x,a}\Psi} }_{
\displaystyle =: U_{n,m,s}}\!\!\biggr|
\le \frac{1}{s} \, W^{x,a}\,.
$$
Since our stopping time satisfies $\tau_{m,s} \ge m$, the random variable 
$U_{n,m,s}$  (depending also on $\varphi$ and $(x,a)$) is
independent of the basic random mappings $F_1, \dots, F_m\,$. (Recall that
the $F_k$ that appear in $S_{\tau,n}^{x,a}$ are such that $k\ge \tau+1$.)
We get
$$
\lim_{s \to \infty} \limsup_{n \to \infty} |V^a\varphi - U_{n,m,s}| = 0
$$
on $\bar\Omega^a_{\varphi}$. Therefore also $V^a\varphi$ is independent
of $F_1, \dots, F_m\,$. This holds for every $m$. By Kolmogorov's
0-1-law, $V^a\varphi$ is almost surely constant.

Note that in the lattice case, the proof simplifies, because we can just
take the first return times of $A_{0,n}$ to $1$.
\end{proof}

\begin{thm}\label{thm:divide} 
Given the random i.i.d. Lipschitz mappings $F_n\,$, let $A_n$ and $B_n$ be as in 
\eqref{eq:lip-an-bn}. Suppose that besides \eqref{eq:assume1} and
\eqref{eq:nofix} [non-degeneracy] and \eqref{eq:assume2} [moment conditions],
 also 
\eqref{eq:assume3} holds, and that
$\Prob\bigl[\hat d\bigl(\wh X_n^{x,a}, \hat o\bigr) \to \infty\bigr]=0$.
Then the SDS induced by the $F_n$ on $\Xx$ has an invariant Radon 
measure $\nu$ that is unique up to multiplication with constants. 

Also, the shift $\wh T$ on 
$\bigl(\wh \Xx^{\N_0}, \Bf(\wh \Xx^{\N_0}), \Prob_{\lambda}\bigr)$ 
is ergodic, where $\lambda$ is the extension of $\nu$ to $\wh \Xx$ and
$\Prob_{\lambda}$ the associated measure on $\wh \Xx^{\N_0}$. 
\end{thm}

\begin{proof} 
Let $\varphi \in\Ccal_c^+(\wh\Xx^{\ell})$. Recall that the function
$\vb_{\varphi} = \Ex_{\lambda}(\varphi\,|\,\If)/\Ex_{\lambda}(\Psi\,|\,\If)$ on
$\wh\Xx^{\N_0}$ is $\wh T$-invariant. For the random variables
$V^{x,a}_{\varphi} = V^a_{\varphi}$, this means that for almost every $a>0$, 
$$
V^a_{\varphi} = V^{A_{0,n}a}_{\varphi} \quad \Pr\text{-almost surely for all}\; n\,.
$$ 
By Proposition \ref{pro:constant}, these random variables are constant
on a set $\bar \Omega_{\varphi}^a \subset\Omega_{\varphi}^a$ with 
$\Pr(\bar \Omega_{\varphi}^a) = 1$. 
Fix one $a_0 > 0$ for which this holds. 

In the lattice case, since we have chosen the maximal $\kappa$ for which 
$\log A_n \in \kappa\cdot \Z$ a.s., the associated centered random walk 
$\log A_{0,n}$ is recurrent on $\kappa\cdot \Z\,$: for every starting point 
$a \in \exp(\kappa\cdot\Z)$, we have that $(A_{0,n}a)_{n \ge 0}$ visits $a_0$
almost surely. We infer that $V^a_{\varphi} = V^{a_0}_{\varphi}$ 
$\Prob$-almost surely for every $a \in\exp(\kappa\cdot\Z)$.   

In the non-lattice case, the multiplicative random walk $(A_{0,n}a)_{n \ge 0}$
starting at any $a > 0$ is topologically recurrent on $\R^+_*$. This means that
for every $a > 0$, with probability $1$ there is a random sequence $(n_k)_{k \ge 0}$
such that $A_{0,n_k}a \to a_0$ as $k \to \infty$.  Proposition \ref{pro:delta}
yields that $V^a_{\varphi} = V^{a_0}_{\varphi}$ on a set 
$\wt\Omega_{\varphi}^a \subset \Omega_{\varphi}^{a_0}$ with probability $1$.

Now let $\{a_k : k \in \N\}$ be dense in $\R^+_*$ and such that 
$\Pr(\wt \Omega_{\varphi}^{a_k}) = 1$ for all $\N$. Using 
Proposition \ref{pro:delta} once more, we get that for every $a > 0$,
$V^a_{\varphi} = V^{a_k}_{\varphi}= V^{a_0}_{\varphi}$ on 
$\bigcap_k \wt \Omega_{\varphi}^{a_k}$. 

We conclude that $\vb_{\varphi}$ is constant $\Prob_{\lambda}$-almost surely.

This is true for any $\varphi \in\Ccal_c^+(\wh\Xx^{\ell})$. Therefore $\wh T$
is ergodic. It follows that up to multiplication with constants,
$\lambda$ is the unique invariant measure on $\wh \Xx$
for the extended SDS, so that $\nu$ is the unique invariant measure on $\Xx$
for the original SDS. By Lemma \ref{lem:inv-exist}(b), $\supp(\nu) = \Ll$. 
\end{proof}

We remark that by projecting, also the shift $T$ on  
$\bigl(\Xx^{\N_0}, \Bf(\Xx^{\N_0}), \Prob_{\nu}\bigr)$ is ergodic. 

\section{The reflected affine stochastic recursion}\label{sec:ref-aff}

We finally consider in detail the SDS of \eqref{eq:aff-ref}. 
Thus, $F_n(x)=|A_nx - B_n|$, so that $\lp(F_n)=A_n$ and 
$d\bigl(F_n(0),0\bigr) = |B_n|$.
%We shall primarily deal with the case when $B_n \ge 0$. 
We assume \eqref{eq:assume1}.  

In the case when $\Ex(\log A_n) < 0$, we can once more apply 
Propositions \ref{pro:furst}, resp. \ref{pro:contract}, and Corollary
\ref{cor:contract}. 

\begin{cor}\label{cor:negdrift} 
If $\Ex(\log^+ A_n) < \infty$ and $-\infty \le \Ex (\log A_n) < 0$ then the 
reflected affine stochastic recursion is strongly contractive on  $\R^+$. 

If in addition
$\Ex(\log^+ |B_n|) < \infty$ then it has a unique invariant probability 
measure $\nu$ on $\R^+$, and it is (positive) recurrent on 
$\Ll = \supp(\nu)$.
\end{cor}

\emph{From now on, we shall be interested in the case when 
$\log A_n$ is centered.}
\\[5pt]
For the time being, we shall only deal with the case when $B_n > 0$.
We can use Remark \ref{rem:non-contractive}; compare with the
arguments used after Corollary \ref{cor:contract}. Thus, the reflected
affine stochastic recursion is topologically irreducible on the set
$\Ll$ given by Corollary \ref{cor:support}. Here, we shall not investigate
the nature of $\Ll$ in detail. It may be unbounded or compact.
%, and even finite.

Since we have $\Xx = \R^+$, the extended space $\wh \Xx$ is just
the first quadrant with hyperbolic metric, and if $f(x) = |ax-b|$ then
$\hat f(x,y) = (|ax-b|, ay)$. We can apply Corollary \ref{cor:transient}
to the extended process.

\begin{pro}\label{pro:local}
Assume that \eqref{eq:assume1} and \eqref{eq:nofix} hold, 
$\Ex(|\log A_n|) < \infty\,$, 
$\Ex(\log A_n) = 0$, $B_n > 0$ almost surely, and $\Ex(\log^+ B_n) < \infty\,$.

\smallskip

If the extended process $(\wh X_n^{x,a})$ is 
conservative, then the normalized distances $D_n(x,y)$ of \eqref{eq:Zn} 
satisfy \eqref{eq:assume3}, that is,
$\Prob[d(Z_n^x,Z_n^y) \to 0] = 1$ for all $x, y \in \Xx$,
where $Z_n^x = X_n/A_{0,n}$.
\end{pro}

\begin{proof}
We have the recursion $Z_0^x = x$ and $Z_n^x = | Z_{n-1}^x - B_n/A_{0,n}|$.
We start with a simple exercise whose proof we omit. Let $c_j > 0$ and
$f_j(x) = |x - c_j|$, $j=1, \dots, s$. 
Then 
\begin{equation}\label{eq:cj}
f_s \circ \dots \circ f_2 \circ f_1(x) \le \max \{ c_1\,, \dots, c_s\}
\quad \text{for all } \; x \in [0\,,\,c_1 + \dots + c_s]\,.
\end{equation}
We prove that for every $\ep > 0$ and $M > 0$ there is $N$ such that 
$$
\Prob(\Gamma_{M, N,\ep}) > 0\,,\quad\text{where}\quad 
\Gamma_{M,N,\ep} = [ D_N(x,y) < \ep  \; \text{for all $x,y$ with}\; 0 \le x, y\le M]\,.
$$
To show this, let $\mu$ be the probability measure on $\R^+_* \times \R^+_*$
governing our SDS, that is, $\Prob[(A_k,B_k) \in U] = \mu(U)$ for any
Borel set $U \subset  \R^+_* \times \R^+_*$.
By our assumptions, there are $(a_1,b_1), (a_2,b_2) \in \supp(\mu)$, 
such that $0 < a_1 < 1 < a_2$ and $b_1, b_2 > 0$. 
We choose $\Delta > 1$ such that $a_1\,\Delta< 1 < a_2/\Delta$
and $b_* =  \min\{b_1,b_2\}/\Delta> 0$, and
we set %$a_* = a_2/\Delta$, $a^* = a_2\,\Delta$ and
$b^* = \max\{b_1,b_2\}\,\Delta$. 

Let $r, s \in \N$. For $k=r+1, \dots, r+s$, we recursively
define indices $i(k) \in \{1,2\}$ by 
$$
i(r+1) = 1, \quad i(k+1) 
= \begin{cases} 1\,,&\text{if }\; a_{i(r+1)} \cdots a_{i(k)} \ge 1,\\
                2\,,&\text{if }\; a_{i(r+1)} \cdots a_{i(k)} < 1 .
  \end{cases}
$$
Therefore $a_1 \le a_{i(r+1)} \cdots a_{i(k)} \le a_2$ for all $k > r$.
We have 
$$
\begin{aligned}
\Prob[a_2/\Delta^{1/r} \le A_k \le a_2\,\Delta^{1/r}
\;\text{and}\; b_* \le B_k \le b^*] &> 0\,,\;\; k=1,\dots,r\,,\AND\\
\Prob[a_{i(k)}/\Delta^{1/s} \le A_k \le a_{i(k)}\,\Delta^{1/s}
\;\text{and}\; b_* \le B_k \le b^*] &> 0\,,\;\; k = r+1, \dots, r+s.
\end{aligned}
$$
Since the $(A_k,B_k)$ are i.i.d., we also get that with positive probability,
$$
\begin{aligned} 
\frac{a_2^k}{\Delta} \le A_{0,k} \le a_2^k\,\Delta&\quad  
\text{for }\; k=1,\cdots,r\,,\\
\frac{a_1}{\Delta} \le A_{r,r+j} \le a_2\,\Delta&\quad  
\text{for }\; j=1,\cdots,s\,,\\
b_* \le B_k \le b^*&\quad  
\text{for }\; k=1,\cdots,r+s\,,
\end{aligned}
$$
and thus, again with positive probability, 
\begin{equation}\label{eq:ineq}
\begin{gathered} 
\frac{B_k}{A_{0,k}} \le \frac{b^*\,\Delta^2}{a_2} \quad 
\text{for }\; k=1,\cdots,r \AND \\
\frac{b_*}{a_2^{r+1}\,\Delta^2} \le 
\underbrace{\frac{B_{r+j}}{A_{0,r+j}}}_{\displaystyle =:c_j} \le 
\frac{b^*\,\Delta^2}{a_1 a_2^r}\quad  \text{for }\;j=1,\cdots,s\,.
\end{gathered}
\end{equation}
We now set $M' = b^*\,\Delta^2/a_2$ and then choose $r$ and $s$ sufficiently 
large such that 
$$
\frac{b^*\,\Delta^2}{a_1 a_2^{r}} < \ep 
\AND  s\,\frac{b_*}{a_2^{r+1}\,\Delta^2} \ge M +  M'.
$$
We set $N = r+ s$ and let $\Gamma_{M, N,\ep}$ be the event on which the inequalities
\eqref{eq:ineq} hold. On $\Gamma_{M, N,\ep}\,$, we can use
\eqref{eq:cj} to get $Z_r^0 \le M'$. 
Since $D_n(x,y)$ is decreasing in $n$, we have for 
$x \in [0\,,\,M]$ that $|Z_r^x - Z_r^0| \le x \le M$ and thus 
$\xi = Z_r^x  \in [0\,,\,M + M']$.
Now we can apply \eqref{eq:cj} with $c_j$ as in 
\eqref{eq:ineq} and obtain $\max_j {c_j} < \ep$ and $c_1 + \dots c_s \ge M +M'$.
But for the associated mappings $f_1, \dots, f_s$ according to \eqref{eq:cj},
we have $Z_n^x = f_s \circ \dots \circ f_1(\xi)$.
We see that on the event $\Gamma_{M, N,\ep}$, one has $Z_n^x < \ep$ for all 
$x \in [0\,,\,M]$, whence $D_N(x,y) < \ep$ for all $x, y \in [0\,,\,M]$.

We can use Lemma \ref{lem:Dnto0} to conclude.
\end{proof}

Combining the last proposition with theorems \ref{thm:lip-trans} and
Theorem \ref{thm:divide}, we obtain the main result of this section.

\begin{thm}\label{thm:refl-aff} 
Consider the reflected affine stochastic recursion \eqref{eq:aff-ref} with
\hbox{$A_n\,, B_n > 0$.} Suppose 
\\[3pt]
\emph{(1)}
non-degeneracy: $\Prob[A_n = 1] < 1$ and $\Prob[A_n x + B_n = x] < 1$
for all $x \in \R$
\\[3pt]
\emph{(2)} moment conditions: $\Ex(|\log A_n|^2) < \infty$ and
$\Ex\bigl((\log^+ B_n)^{2+\ep}\bigr) < \infty$ for some $\ep > 0$
\\[3pt]
\emph{(3)} centered case: $\Ex(\log A_n) = 0$.
\\[3pt]
Then the SDS has a unique invariant Radon measure $\nu$ on $\R^+$,
it is topologically recurrent on $\Ll = \supp(\nu)$. 
The time shift on the trajectory space 
$\bigl((\R^+)^{\N_0}, \Prob_{\nu}\bigr)$ is ergodic.
%, where 
%$\Prob_{\nu} = \int_{\R^+} \Prob_x\,d\nu(x)$ and $\Prob_x$ is
%the probability measure governing the process starting at $x$.
\end{thm}     

We now answer the additional question when there is an invariant
\emph{probability} measure, i.e., when $\nu(\Ll)< \infty$.

\begin{thm}\label{thm:finite-meas} In the situation of Theorem \ref{thm:refl-aff},
suppose also that $\Ex(|\log A_n|^{2+\ep}) < \infty$ and 
$\Prob[B_n \ge b] = 1$ for some $b > 0$.
Then we have $\nu(\Ll) < \infty$ if and only if the set $\Ll$ is bounded.
\end{thm}

The proof will be based on the next proposition, which may be of
interest in its own right.

\begin{pro}\label{pro:ret} 
For any $x, t \ge 0$, let
$$
\tau^{[0,t)}_x = \inf \{n \ge 1 : X_n^x < t \}
$$
be the time of the first visit in the interval $[0\,,\,t)$. Under the assumptions of 
Theorem \ref{thm:finite-meas}, there is $x(t) > 0$ such that
for all $x \ge x(t)$, one has 
$$
\Ex(\tau^{[0,t)}_x)=\infty.
$$
\end{pro}

\begin{proof}
Consider the affine recursion without reflection $Y_n^x = A_nY_{n-1}^x - B_n$. 
If $Y_k^x \ge t$ for $k=1, \dots, n$ then 
$X_k^x = Y_k^x$ for those $k$, and then we have $\tau^{[0,t)}_x > n$. That is,
$$
\Prob[\tau^{[0,t)}_x > n] \ge \Prob[Y_k^x \ge t\,,\;k=1, \dots, n].
$$
We have 
\begin{equation}\label{eq:expli}
Y_k^x \ge t \iff 
\underbrace{\sum_{j=1}^k \frac{B_j}{A_{0,j}}}_{\textstyle \check R_k^0} 
+ \frac{t}{A_{0,k}} \le x\,.  
\end{equation}
Now consider the affine stochastic recursion generated by the inverses
of the affine mappings $F_n(x) = A_n x - B_n$. These are
$$
\check F_n(y) = \check A_n y + \check B_n\,,\quad\text{where}\quad
\check A_n = 1/A_n \AND \check B_n = B_n/A_n\,.
$$
They satisfy moment conditions of the same order as $A_n$, resp. 
$B_n\,$, so that the associated affine recursion $(\check Y_n^y)$
is recurrent on the support of its unique invariant measure. Thus,
there is $u > 0$ (sufficiently large) such that
$\Prob[\check Y_n^y \le u \;\text{infinitely often}] = 1$
for any starting point $y$.
The right process induced by the $\check F_n$ is 
$\check R_n^y = \check F_1 \circ \dots \circ \check F_n(y)$. 
It is not a Markov
chain, but $\check R_n^y$ has the same distribution as $\check Y_n^y\,$.
In particular, $\check R_k^0$ appears above in \eqref{eq:expli},
and 
$$
\sum_n \Prob[\check R_n^0 \le u] = \sum_n \Prob[\check Y_n^0 \le u]
=\infty.
$$
Now, if $\check R_n^0 \le u$, then for $k=1, \dots, n$,
$$
\check R_k^0 + \frac{t}{A_{0,k}} \le \check R_n^0 
+\underbrace{\frac{B_k}{A_{0,k}}}_{\displaystyle \le \check R_n^0}\frac{t}{B_k}  
\le u(1+t/b) =: x(t).
$$
If $x \ge x(t)$ then we see that 
$$ 
\Prob[\check R_n^0 \le u] \le \Prob[Y_k^x \ge t\,,\;k=1, \dots, n].
$$
Therefore 
$$
\sum_n \Prob[\tau^{[0,t)}_x > n] \ge \sum_n \Prob[\check R_n^0 \le u]\,,
$$ 
and the statement follows.
\end{proof}

\begin{proof}[Proof of Theorem \ref{thm:finite-meas}] Suppose that $\Ll$ is unbounded. We use the distinction between
positive and null recurrence as in Corollary \ref{cor:returntime}.
We fix a suitable $t > 0$ such that the interval $[0\,,\,t)$ intersects 
$\Ll$. We consider the probability measure 
$\nu_t = \frac{1}{\nu([0,t))} \nu|_{[0,t)}$ and the SDS
$(X_n^{\nu_t})$ with initial distribution $\nu_t\,.$ We shall show that 
its return time $\tau^{[0\,,\,t)}$ to $[0\,,\,t)$ has infinite 
expectation. Then $\nu$ cannot be finite. 

We know that there is $u \in \Ll$ with $u > x(t)\,$, with $x(t)$ as in
Proposition \ref{pro:ret}. We let $U$ be an open interval that
contains $u$ and does not intersect $[0\,,\,t]$. We apply Theorem 
\ref{thm:recur} to a starting point $x_0 \in [0\,,\,t)\cap \Ll$. 
There is $m$ such that $\Prob[X_m^{x_0} \in U]>0$. This means that there
are $f_1\,,\dots, f_m \in \supp(\wt \mu)$ such that 
$f_m \circ \dots \circ f_1(x_0) \in U$. 
(Each $f_j$ is of the form $f_k(x) = |a_j x - b_j|$.)
There must be a maximal $k < m$
for which $x_k = f_k \circ \dots \circ f_1(0) \in [0\,,\,t]$. 
Note that $x_j \in \Ll$ for all $j$ by Corollary \ref{cor:support},
compare with Remark \ref{rmk:hypotheses}(b).

We now may assume without loss of generality that $k=0$.
Therefore we can find neighbourhoods (open intervals) 
$U_0, U_1, \dots, U_{m-1}, U_m=U$ of the respective $x_j$
such that $U_0 \subset [0\,,\,t)$, while $U_j \cap [0\,,\,t) = \emptyset$
for $j >0$, and 
$$
\wt\mu(\{ f : f(U_{j-1}) \subset U_j \}) > 0\,,\quad j=k+1, \dots, m.
$$  
This translates into
$$
\Prob(\Lambda_x) \ge \alpha > 0\quad \text{for all}\; x \in U_0\,,\quad 
\text{where}\quad
\Lambda_x = [X_j^x \in U_j\,,\; j = 1, \dots, m].
$$ 
So we can now consider the SDS starting at $x \in U_0$, leaving
$(0\,,\,t]$ at the first step, and reaching some $y \in U$ in $m$ steps.
After that, it takes $\tau^{[0,t)}_y$ steps to return to $(0\,,\,t]$.
We formalize this, and remember that $U_j \cap \Ll \ne \emptyset$
for every $j$. Just for the purpose of the next lines, we consider
the measure $\sigma_x(B) = \Prob(\Lambda_x \cap [X_m^x \in B])$, 
where $x \in U_0$. It is concentrated on $U$ with 
$\sigma_x(U) \ge \alpha$, and
$$
\Ex(\tau^{[0,t)}) \ge \int_{U_0} 
\Ex(\tau^{[0,t)}_x \cdot \uno_{\Lambda_x}) \,d\nu_t(x)
\ge   
\int_{U_0} \underbrace{\Bigl(\int_U \bigl( m + \Ex(\tau^{[0,t)}_y )\bigr)
d\sigma_x(y) \Bigr)}_{\displaystyle = \infty\;
\text{by Proposition \ref{pro:ret}}} \,d\nu_t(x) = \infty.
$$
Therefore $\nu$ must have infinite mass.
\end{proof}

We now discuss an example.

\begin{exa}\label{exa:simple}
We let $0 < p < 1$ and
$$
A_n = \begin{cases} 2 \;&\text{ with probability }\; p\,, \\
1/2 \;& \text{ with probability }\; q=1-p\,,
\end{cases}
    \qquad 
B_n=1 \; \text{ always.}  
$$
Thus, we randomly iterate the transformations $f_1(x) = |2x-1|$ and 
$f_{-1}(x) = |x/2-1|$. In other words, $F_n(x) = |2^{\epb_n}x - 1$,
where $(\epb_n)_{n \ge 1}$ is a sequence of i.i.d. $\pm 1$-valued
random variables with $\Prob[\epb_n=1] = p$ and $\Prob[\epb_n=-1] = q$.
%The SDS is strongly contractive when $p < 1/2$, and we are in the 
%log-centerd case when $p=1/2$. 

Keeping in mind Remark \ref{rmk:hypotheses}(b), we now determine $\Ll$
as the smallest non-empty closed set which satisfies
$f_{\pm 1}(\Ll) \subset \Ll$. First of all, we see that each of the two 
functions maps the interval $[0\,,\,1]$ into itself. Thus, we must have 
$\Ll \subset [0\,,\,1]$. 

Let $\alpha = \max \Ll$. Then $\alpha \ge 2/3$, because $2/3 \in \Ll$
as the attracting fixed point of $f_{-1}\,$.
We must have 
$(1 + \alpha)/2 = f_{-1}\circ f_1 \circ f_{-1}(\alpha)\in \Ll$,
whence it is $\le \alpha$. Therefore $\alpha = 1$. We get that
$1 \in \Ll$. The set of all iterates
of $1$ under $f_{\pm 1}$ is
$$
\{ f_{i_1} \circ \dots \circ f_{i_n}(1) : n \ge 0\,,
\; i_j = \pm 1 \} = \D\,, \quad\text{where}\quad
\D = \Z[\tfrac12] \cap [0\,,\,1]\,,
$$
and  $\Z[\frac12]$ stands for the dyadic rationals, i.e., 
rationals whose denominator is a power of $2$.
Since $\D$ is dense, $\Ll = [0\,,\,1]$. 
\\[5pt]
\underline{Contractive case} ($p < 1/2$). We can apply Corollary 
\ref{cor:negdrift} and get a unique invariant probability measure
$\nu$, which is supported on $[0\,,\,1]$. 
\\[5pt]
\underline{Log-centered case} ($p = 1/2$).
Since $\Ll$ is compact, the extended SDS is clearly conservative.
In particular, $D_n(x,y) \to 0$ almost surely for all $x,y$.
We now undertake an additional effort to clarify that the SDS
is \emph{not} locally contractive.

For the symmetric random walk $S_n = \epb_1 +\dots + \epb_n$ on 
$\Z$, let $M_n = \max\{0, S_1\,,\dots, S_n\}$. Now consider our SDS
$(X_n^x)_{n\ge 0}$ with $x \in [0\,,\,1]$. It is an instructive
exercise to prove the following by induction on $n$. 

\begin{lem}\label{lem:iterates} The map $x \mapsto X_n^x$ is continuous
and piecewise affine and continuous on $[0\,,\,1]$, and
there are random variables $\deb \in \{ -1,1\}$ and $C_j = C_{j,M_n} \in \Z[\frac12]$
such that
$$
X_n^x = (-1)^j\,\deb\,2^{S_n} x + C_j \quad \text{on } \;
I_{j,M_n}\,,\quad \text{where }\;
I_{j,k} = [(j-1)2^{-k}\,,\,j2^{-k}]\,,\; j = 1, \dots, 2^k\,.
$$
In particular, the images of each of the intervals $I_{j,M_n}$ under 
$x \mapsto X_n^x$ coincide and have the form
$$
[(L_n-1)/2^{M_n - S_n}\,, L_n /2^{M_n - S_n}]\,,
$$ 
where $L_n$ is an integer random variable with 
$1 \le L_n \le 2^{M_n - S_n}$. 
\end{lem}
Recall the \emph{strictly ascending ladder epochs} of the random walk
$(S_n)$, 
$$
\tb(0)  = 0 \AND \tb(k+1)  = \inf \{ n > \tb(k) : S_n > S_{\tb(k)} \}\,.
$$
They are all a.s. finite, and $S_{\tb(k)} = M_{\tb(k)} = k$. 
By Lemma \ref{lem:iterates}, the image of each interval $I_{j,k}$ is the whole
of $[0\,,\,1]$. From this and the specific form that $x \mapsto X_n^x$ has to
take, one sees that the only two choices for the mapping 
$x \mapsto X_{\tb(k)}^x$ are
$$
X_{\tb(k)}^x = f_1^{(k)}(x) \quad\text{or}\quad
X_{\tb(k)}^x = 1- f_1^{(k)}(x)\,,
$$
where $f^{(k)}$ denotes the $k$-th iterate of the function $f$.
Therefore, considering the fixed points $x_0= 1$ and $y_0 = 1/3$ of $f_1$,
we get
$$
|X_{\tb(k)}^{x_0} - X_{\tb(k)}^{y_0}| = 2/3 \quad \text{for all}\; k.
$$
Thus, we do not have local contractivity.
\\[5pt]
\underline{Expanding case} ($p > 1/2$).
Since $\Ll$ is compact, the SDS is conservative for any value of $p$,
so that there are always invariant probability measures. We show that
in the expanding case, there are infinitely many mutually singular
ones.
Fix $r$, an odd prime or $r=1$, and define %%$\D_r$ as follows.
$$
%\D_1 =\D, \quad \text{and if }\; r > 1,\quad
\D_r = \left\{ \frac{k}{r\,2^n} : k, n \in \N_0\,,\; %%k \in \N\,,\;
               k \le r\,2^n\,,\; \lcd(k,r\,2^n) = 1\right\}\,.
$$	          
(Note that we must have $0 < k < r\,2^n$ when $r > 1$.)
Then it is easy to verify that $f_{\pm 1}(\D_r) \subset \D_r\,$.
Thus, when we start at a point $x \in \D_r\,$, then
$(X_n^x)$ can be seen as a Markov chain on the denumerable state space
$\D_r\,$. Let $p(x,y)= \Prob[X_1^x = y]$ denote its transition matrix. 
It is not hard to verify that it is irreducible (all states communicate), although
we do not really need this. 
We partition 
$\D_r = \bigcup_n \D_{r,n}\,$, where $\D_{r,n}$ consists of all $\frac{k}{r\,2^n}$
as above with the specific value of $n$. If $n \ge 1$, then we see
that for each $x \in \D_{r,n}\,$, we have that
$$
p(x,\D_{r,m}) = \sum_{y \in \D_{r,m}} p(x,y) 
= \begin{cases} p\,,&\text{if}\; m = n-1\,,\\
                q\,,&\text{if}\; m = n+1\,,\\
                0\,,&\text{otherwise.}\\
  \end{cases}
$$
A similar identity for $x \in \D_{r,0}$ does not hold, so that we cannot
define the factor chain on $\N_0\,$. Nevertheless, since each $\D_{r,n}$ is finite, 
we can use comparison with the birth-and-death chain on $\N_0$ with transition
probabilities $\bar p(n,n+1) = q$ and $\bar p(n,n-1) = p$ for $n \ge 1$.
(We do not need to specify the outgoing probabilities at $0$.) 
Thus, our Markov chain on $\D_r$ is positive recurrent when $p > 1/2$,
null recurrent when $p=1/2$, and transient when $p < 1/2$.
In particular, when $p > 1/2$, it has a unique invariant probability
measure $\nu_r$ on the countable set $\D_r\,$. Since it is a probability
measure, we can lift it to a Borel measure on $[0\,,\,1]$ by setting
$\nu_r(B) = \sum_{x \in \D_r \cap B} \nu_r(x)$. 
Thus, each $\nu_r$ is also an invariant probability measure for the
(``topological'') SDS on $[0\,,\,1]$, and all the $\nu_r$ are pairwise mutually 
singular.
\end{exa}

\begin{rmk}\label{rmk:dyadic}
Regarding the last example, we underline that the respective discrete,
denumerable Markov chains on $\D_r$ have precisely the opposite behaviour of
the SDS on $[0\,,\,1]$: the Markov chain is transient precisely when the
SDS is strongly contractive (and positive recurrent), and it is null recurrent
precisely when the SDS is weakly, but not strongly contractive (and null-recurrent).
But this fact should not be surprising. Indeed, let us compare this with
the affine stochastic recursion $Y_n^x = 2^{L_n} x + B_n$, where $(L_n,B_n)$
are $2$-dimensional i.i.d. random variables with $L_n \in \Z$ and 
$B_n \in \Z[\frac12]$. If the starting point $x$ is also a
dyadic rational, then we can consider $(Y_n^x)$ as an SDS both on $\R$ with
Euclidean distance and on the field $\Q_2$ of dyadic numbers with the distance 
induced by the dyadic norm. Under the usual moment conditions, this SDS is 
transient on 
$\R$ precisely when it is strongly contractive on $\Q_2\,$, and weakly (but not
strongly) contractive on $\R$ precisely when it has the same property on
$\Q_2\,$.  
\end{rmk}

In conclusion, we briefly touch another example, considering only 
the log-centered case.

\begin{exa}\label{exa:less-simple}
We let $0 < p < 1$ and
$$
A_n = \begin{cases} 3 \;&\text{ with probability }\; 1/2\,, \\
1/3 \;& \text{ with probability }\; 1/2\,,
\end{cases}
    \qquad 
B_n=1 \; \text{ always.}  
$$
This time, we randomly iterate  $g_1(x) = |3x-1|$ and 
$g_{-1}(x) = |x/3-1|$.
A brief discussion shows that the limit set must be unbounded:
suppose that $\alpha = \sup \Ll < \infty$.
Then we must have $g_{i_n} \circ \dots \circ g_{i_1}(\alpha) 
\in \Ll$ for any choice of $n$ and $i_j \in \{-1,1\}$
($j=1, \dots, n$). But for any $\alpha$ we can find some choice where 
$g_{i_n} \circ \dots \circ g_{i_1}(\alpha) > \alpha$, a
contradiction. 

Thus, the invariant Radon measure has infinite mass.
\end{exa}

A more detailed study of these and similar classes of reflected
affine stochastic recursions are planned to be the subject of 
future work.

\end{document}